\documentclass[a4paper,12pt]{article}
\usepackage[dvipsnames,table,svgnames]{xcolor}
\usepackage[utf8]{inputenc}
\usepackage{verbatim}
\usepackage[margin=3cm]{geometry}
\usepackage[algo2e,ruled,vlined]{algorithm2e}
\usepackage{algorithm}
\usepackage{here}
\usepackage{setspace}
\usepackage[normalem]{ulem}
\usepackage{algpseudocode}
\usepackage[T1]{fontenc}
\usepackage{dsfont}
\usepackage{url}
\usepackage[inline]{enumitem}
\usepackage{microtype}
\usepackage[affil-sl]{authblk}
\usepackage{graphicx}
\usepackage{nicematrix}
\usepackage{capt-of}

\usepackage{tikz}
\usetikzlibrary{tikzmark}
\usetikzlibrary{positioning}
\usetikzlibrary{graphs}
\usetikzlibrary{graphs.standard}
\usetikzlibrary{decorations.pathmorphing}
\usetikzlibrary{decorations.text}
\usetikzlibrary{patterns,shadings}
\usetikzlibrary{matrix}
\pgfdeclarelayer{bg}    
\pgfsetlayers{bg,main}  

\usepackage{amsmath}
\usepackage{amssymb}
\usepackage{amsthm}
\usepackage{mathtools}
\usepackage{nicefrac}
\usepackage{amsfonts,latexsym,graphics}
\usepackage{color}
\usepackage{colortbl}
\usepackage[hidelinks]{hyperref}
\usepackage{multirow}
\usepackage[english]{babel}
\usepackage{epsfig}
\usepackage{bm}
\usepackage{makeidx}
\usepackage[labelformat=simple]{subcaption}
\usepackage{adjustbox}

\usepackage{xspace}         

\theoremstyle{definition}
\newtheorem{definition}{Definition}

\theoremstyle{plain}

\newtheorem{lemma}[definition]{Lemma}
\newtheorem{theorem}[definition]{Theorem}

\newcommand*{\claimproofname}{Proof}

\newtheorem{caseinner}{Case}

\newenvironment{case}[1]{%
	\renewcommand\thecaseinner{#1}%
	\caseinner
}{\endcaseinner\vspace{-5pt}}

\renewcommand\thesubfigure{(\alph{subfigure})}

\def\matrix0{{\mbox {\boldmath $O$}}}

\def\x{{\mbox{\boldmath $x$}}}

\def\vec0{\mbox{\bf 0}}

\newcommand\tran{\mkern-2mu\raise1.25ex\hbox{$\scriptscriptstyle\top$}\mkern-3.5mu}

\def\G{\Gamma}

\def\set#1{\lbrace#1\rbrace}

\DeclareMathOperator{\Aut}{Aut}

	{\end{list}}

\newenvironment{dense_enumerate}[1]{%
	\begin{enumerate}[itemsep=0.2ex,parsep=0.2ex,partopsep=0.2ex,topsep=0.2ex,#1]%
	}
	{\end{enumerate}}

\newcounter{constraint}
\renewcommand{\theconstraint}{\alph{constraint}}
\newcommand{\constraintlabel}[1]{\refstepcounter{constraint}\label{#1}(\theconstraint)}

\SetArgSty{}
\SetKwFor{For}{for}{do}{}
\SetKwInOut{Input}{Input}
\SetKwInOut{Output}{Output}

\makeatletter
\let\c@figure\c@table
\let\ftype@figure\ftype@table
\let\ext@figure\ext@table

\makeatother 

\definecolor{darkorange}{RGB}{255,132,8}

\title{On the existence of small strictly Neumaier graphs}

\author{Aida Abiad\thanks{\texttt{a.abiad.monge@tue.nl}, Department of Mathematics and Computer Science, Eindhoven University of Technology, The Netherlands\newline Department of Mathematics: Analysis, Logic and Discrete Mathematics, Ghent University, Belgium\newline Department of Mathematics and Data Science of Vrije Universiteit Brussel, Belgium} \quad Maarten De Boeck\thanks{\texttt{mdeboeck@memphis.edu},  Department of Mathematical Sciences, University of Memphis, U.S.A.\newline Department of Mathematics: Algebra and Geometry, Ghent University, Flanders, Belgium} \quad  Sjanne Zeijlemaker\thanks{\texttt{s.zeijlemaker@tue.nl},  Department of Mathematics and Computer Science, Eindhoven University of Technology, The Netherlands}}

\date{}

\begin{document}

	\maketitle

	\begin{abstract}
A Neumaier graph is a non-complete edge-regular graph containing a regular clique. In this work, we prove several results on the existence of small strictly Neumaier graphs. In particular, we present a theoretical proof of the uniqueness of the smallest strictly Neumaier graph with parameters $(16,9,4;\allowbreak 2,4)$, we establish the existence of a strictly Neumaier graph with parameters $(25,12,5;\allowbreak 2,5)$, and we disprove the existence of strictly Neumaier graphs with parameters $(25,16,9;3,5)$, $(28,18,11;4,7)$, $(33,24,17;6,9)$, $(35,22,12;3,5)$, $(40,30,\allowbreak 22;\allowbreak7,10)$ and $(55,34,18;3,5)$. Our proofs use combinatorial techniques and a novel application of integer programming methods.		
\end{abstract}

	\section{Introduction}

	A regular graph is called \emph{edge-regular} if any two adjacent vertices have the same number ($\lambda$) of common neighbors, and \emph{co-edge-regular} if every pair of non-adjacent vertices have the same number ($\mu$) of common neighbors. A \emph{strongly regular graph} is an edge and co-edge regular graph. A \emph{regular clique} of order $s$ in a regular graph is a clique having the property that every vertex outside of it is adjacent to the same positive number of vertices of the clique, denoted by $e$. A \emph{Neumaier graph} is a non-complete edge-regular graph containing a regular clique. A Neumaier graph that is not a strongly regular graph is called a \emph{strictly Neumaier graph}, and its parameter set is denoted as $(v,k,\lambda;e,s)$, where~$v$ is the number of vertices and~$k$ the degree.
	\par In his 1981 paper \cite{N1981}, Neumaier studied regular cliques in edge-regular graphs, and he showed that all vertex-transitive, edge-transitive graphs with a regular clique are strongly regular. He subsequently raised the question whether there are edge-regular graphs with a regular clique, that are not strongly regular, i.e.\ whether there are strictly Neumaier graphs. Greaves and Koolen \cite{greaves2018edge} gave an answer to this question by constructing an infinite family of strictly Neumaier graphs. The same authors provided a second construction in \cite{greaves2019another}. All strictly Neumaier graphs described in \cite{greaves2018edge,greaves2019another} have $e=1$. Evans, Goryainov and Panasenko \cite{evans2018smallest} presented a family of strictly Neumaier graphs which is the only known family of strictly Neumaier graphs with $e>1$. Abiad, De Bruyn, D'haeseleer and Koolen \cite{dcc2020} investigated Neumaier graphs with few eigenvalues, and showed that Neumaier graphs with four distinct eigenvalues do not exist, and recently the three present authors, together with Castryck and Koolen, presented several new non-existence results and a new construction of strictly Neumaier graphs with~$e=1$ \cite{ABIAD2023105684}. 
\par In this article we present a variety of results on the existence of small strictly Neumaier graphs which resolve several open cases of parameters sets. In particular, we provide a theoretical proof for the uniqueness of the smallest strictly Neumaier graph, which has parameters $(16,9,4;2,4)$, thus confirming the computational result obtained by Evans, Goryainov and Panasenko \cite{evans2018smallest}. We also show that a strictly Neumaier graph with parameters $(25,12,5;2,5)$ exists, presenting a construction starting from a Latin square. At the time of writing, this is the first strictly Neumaier graph with odd order and~$e>1$. While optimization techniques have been successfully used on other algebraic graph classes such as cages \cite{RB2015} or mixed Moore graphs \cite{LMF2016}, their potential has not yet being exploited in the context of Neumaier graphs. In this regard, we propose an Integer Linear Programming (ILP) model and we use it to disprove the existence of strictly Neumaier graphs with several parameter sets; in particular, we show that $(25,16,9;3,5)$, $(28,18,11;4,7)$, $(33,24,17;6,9)$ $(35,22,12;3,5)$, $(40,30,22;7,10)$ and $(55,34,18;3,5)$ do not admit a strictly Neumaier graph.

	\section{Preliminaries}
	Throughout this paper we will consider simple graphs (i.e., undirected, loopless, no multiple edges). For a graph $\G$ we denote the set of vertices at distance $i$ from a given vertex $u$ by $\G_{i}(u)$. The set of neighbors of $u$, $\G_{1}(u)$, is denoted by $\G(u)$. The degree~$|\G(u)|$ of a vertex~$u$ is denoted by~$d_u$. Adjacency between vertices is denoted by $\sim$ and the set of edges between two sets of vertices~$S$ and~$T$ is denoted by~$E(S,T)$. 
	\par A graph is \emph{($k$-)regular} if each vertex is adjacent to $k$ vertices, for some integer $k$. A regular graph is \emph{($\lambda$-)edge-regular} if it is non-empty, and any pair of adjacent vertices has exactly $\lambda$ common neighbors for some integer $\lambda$; it is \emph{($\mu$-)co-edge-regular} if it is not complete, and any pair of non-adjacent vertices has exactly $\mu$ common neighbors for some integer $\mu$. A graph that is both edge-regular and co-edge-regular is called \emph{strongly regular}. An edge-regular graph with parameters $(v,k,\lambda)$ has $v$ vertices, is $k$-regular and $\lambda$-edge-regular; a co-edge-regular graph with parameters $(v,k,\mu)$ has $v$ vertices, is $k$-regular and $\mu$-co-edge-regular. A strongly regular graph has parameters $(v,k,\lambda,\mu)$ if it is edge-regular with parameters $(v,k,\lambda)$ and co-edge-regular with parameters $(v,k,\mu)$. An edge-regular graph satisfies the following properties.

   \begin{theorem}[{\cite[Section 1.1]{bcn}},{\cite[Corollary 3.2]{ABIAD2023105684}}]\label{th:ERG}
    Let $\Gamma$ be an edge-regular graph with parameters $(v,k,\lambda)$. Then 
    \begin{enumerate}[label=(\roman*),align=left,leftmargin=*]
        \item $v-2k+\lambda\ge 0$,
        \item $\lambda k \equiv 0 \pmod 2$,
        \item $vk\lambda \equiv 0 \pmod 6$,
        \item $(v-k-1)(v-k-2)-k(v-2k+\lambda) \ge 0$.
    \end{enumerate}
   \end{theorem}
	
	Let~$\G$ be a graph with vertex set $V(\G)$ and let~$S\subset V(\G)$. If every vertex in $V(\G) \setminus S$ has precisely $e>0$ neighbors in $S$, we say that $S$ is \emph{$e$-regular}. A \emph{clique} of $\G$ is a subset of $V(\G)$ wherein all vertices are pairwise adjacent; a \emph{coclique} of $\G$ is a subset of $V(\G)$ wherein all vertices are pairwise non-adjacent. 
	\par A graph is a \textit{Neumaier graph} with parameters $(v,k,\lambda;e,s)$ if it is edge-regular with parameters $(v,k,\lambda)$ and has an $e$-regular clique of size $s$. A Neumaier graph which is not strongly regular is called \textit{strictly Neumaier}. Neumaier proved the following results regarding regular cliques in Neumaier graphs.
	
	\begin{theorem}[\cite{N1981}, Theorem 1.1]\label{th:maxcliques}
		Let $\G$ be a Neumaier graph with parameters $(v,k,\lambda;e,s)$. Then 
		\begin{enumerate}[label=(\roman*),align=left,leftmargin=*]
			\item the largest clique of $\G$ has size $s$,
			\item \label{it:allCreg} all regular cliques are $e$-regular,
			\item the regular cliques are exactly the cliques of size $s$.
		\end{enumerate}
	\end{theorem}
	
	Not every choice of parameters~$(v,k,\lambda;e,s)$ admits a (strictly) Neumaier graph. Naturally, the parameters should satisfy $e\le s-1$, $k < v-1$ and $s-2\le \lambda < k$. In addition, they must satisfy the following restrictions.
	
	\begin{theorem}[{\cite[Theorem 1.1]{N1981} and \cite[Theorem 1]{evans2018smallest}}]\label{th:existenceconditions}
		The parameters $(v,k,\lambda;e,s)$ of a Neumaier graph satisfy the following conditions:
		\begin{enumerate}[label=(\roman*),align=left,leftmargin=*]
			\item $k-s+e-\lambda -1 \ge 0$,
			\item $s(k-s+1)=(v-s)e$,
			\item $s(s-1)(\lambda-s+2)=(v-s)e(e-1)$.
		\end{enumerate}
	\end{theorem}
	
	For strictly Neumaier graphs, some additional conditions are known. We refer to \cite[Proposition 5.1]{greaves2018edge}, \cite[Theorem 1.3]{N1981}, \cite[Theorem 4.1]{soicher2015cliques},\cite[Lemma 4.7]{evansphd},   \cite[Theorem 4.10]{evansphd} and \cite[Corollary 3.2, Theorem 3.4]{ABIAD2023105684} for more details.
	
	\begin{theorem}\label{th:existenceconditionsstrict}
		The parameters $(v,k,\lambda;e,s)$ of a strictly Neumaier graph satisfy the following conditions:
		\begin{enumerate}[label=(\roman*),align=left,leftmargin=*]
			\item $s \ge 4$ and, as a result, $\lambda \ge 2$,
			\item $e \le k-2$,
			\item $v \notin \set{2k-\lambda, 2k-\lambda+1}$,
			\item $k-s+e-\lambda -1 \ge 1$.
			\item $(v-k-1)(v-k-2)-k(v-2k+\lambda)>0$,
			\item $(v,k,\lambda;e,s)\neq (6\ell+3,4\ell+2,3\ell;\ell+1,2\ell+1)$ for some integer $\ell\geq 3$.
		\end{enumerate}
	\end{theorem}

    A parameter set is called \emph{admissible} if it satisfies all of the above conditions for strictly Neumaier graphs. Table~\ref{table:neu:parametersstrictlyNG} gives an overview of all known existence results for admissible parameter sets with~$v \le 64$. 
	
    \begin{table}[!htp]
        \small
				\begin{minipage}[t]{.5\linewidth}
					\centering
                    \vspace{0pt}
					\begin{tabular}{c | c | c | c | c | c }
						$v$ & $k$ & $\lambda$ & $e$ & $s$ & Number\\\hline\hline
						\textbf{16} & \textbf{9} & \textbf{4} & \textbf{2} & \textbf{4} & \textbf{1,\cite{evans2018smallest}, Section \ref{sec:neu:uniqueNG16}}\\
                        \hline
						22 & 12 & 5 & 2 & 4 &  0, \cite{evans2018smallest}\\ \hline
						24 & 8 & 2 & 1 & 4 & $\ge 6$, \cite{evans2018smallest,goryainov2014cayley,greaves2019another} \\  \hline
						\multirow[t]{2}{*}{\textbf{25}} & \textbf{12} & \textbf{5} & \textbf{2} & \textbf{5} & $\mathbf{\ge}\mathbf{1}$, \textbf{Section~\ref{sec:neu:new25}} \\
						& \textbf{16} & \textbf{9} & \textbf{3} & \textbf{5} & \textbf{0, Section~\ref{sec:neu:nonexistence}}\\\hline
						26 & 15 & 8 & 3 & 6 & \\ \hline
						\multirow[t]{4}{*}{28} & 9 & 2 & 1 & 4 & $\ge 4$, \cite{evans2018smallest,greaves2018edge}\\
						& \multirow[t]{2}{*}{15} & 6 & 2 & 4 & \\
						&  & 8 & 3 & 7 & \\ 
						& \textbf{18} & \textbf{11} & \textbf{4} & \textbf{7} & \textbf{0, Section~\ref{sec:neu:nonexistence}}\\ \hline
						\textbf{33} & \textbf{24} & \textbf{17} & \textbf{6} & \textbf{9} & \textbf{0, Section~\ref{sec:neu:nonexistence}}\\ \hline
						34 & 18 & 7 & 2 & 4 & \\ \hline
						\multirow[t]{4}{*}{35} & 10 & 3 & 1 & 5 & \\ 
						& 16 & 6 & 2 & 5 & \\ 
						& 18 & 9 & 3 & 7 & \\ 
						  & \textbf{22} & \textbf{12} & \textbf{3} & \textbf{5} & \textbf{0, Section~\ref{sec:neu:nonexistence}}\\ \hline
						\multirow[t]{5}{*}{36} & 11 & 2 & 1 & 4 & \\
						& 15 & 6 & 2 & 6 & \\
						& 20 & 10 & 3 & 6 & \\ 
						& 21 & 12 & 4 & 8 & \\ 
						& 25 & 16 & 4 & 6 & \\\hline
						\multirow[t]{5}{*}{40} & 12 & 2 & 1 & 4 & $\ge 1$, \cite{evans2018smallest}\\ 
						& \multirow[t]{2}{*}{21} & 8 & 2 & 4 & \\ 
						& & 12 & 4 & 10 & \\ 
						& 27 & 18 & 6 & 10 & \\ 
						& \textbf{30} & \textbf{22} & \textbf{7} & \textbf{10} & \textbf{0, Section~\ref{sec:neu:nonexistence}}\\ \hline
						\multirow[t]{3}{*}{42} & 11 & 4 & 1 & 6 & \\ 
						& 21 & 10 & 3 & 7 & \\ 
						& 26 & 15 & 4 & 7 & \\ \hline
						44 & 28 & 18 & 6 & 11 & \\\hline
						\multirow[t]{8}{*}{45} & 12 & 3 & 1 & 5 & \\
						& \multirow[t]{2}{*}{20} & 7 & 2 & 5 & \\
						&  & 10 & 3 & 9 & \\
						& 24 & 13 & 4 & 9 & \\ 
						& \multirow[t]{2}{*}{28} & 15 & 3 & 5 & \\
						&  & 17 & 5 & 9 & \\
						& 32 & 22 & 6 & 9 & \\\hline
						\multirow[t]{3}{*}{46} & 24 & 9 & 2 & 4 & \\
						& 25 & 12 & 3 & 6 & \\
						& 27 & 16 & 5 & 10 & \\\hline
						\multirow[t]{3}{*}{48} & 12 & 4 & 1 & 6 & \\
						& 14 & 2 & 1 & 4 & \\
					\end{tabular}
				\end{minipage}%
				\begin{minipage}[t]{.5\linewidth}
					\centering
                    \vspace{0pt}
					\begin{tabular}{c | c | c | c | c | c }
						$v$ & $k$ & $\lambda$ & $e$ & $s$ & Number\\\hline\hline
						\multirow[t]{3}{*}{49} & 18 & 7 & 2 & 7 & \\
						& 24 & 11 & 3 & 7 & \\
						& 30 & 17 & 4 & 7 & \\
						& 36 & 25 & 5 & 7 & \\\hline
						50 & 28 & 15 & 4 & 8 & \\\hline
						\multirow[t]{2}{*}{51} & 20 & 7 & 2 & 6 & \\\hline
						\multirow[t]{4}{*}{52} & 15 & 2 & 1 & 4 & $\ge 1$, \cite{greaves2018edge}\\
						&\multirow[t]{2}{*}{27} & 10 & 2 & 4 & \\
						& & 16 & 5 & 13 & \\
						& 36 & 25 & 8 & 13 & \\\hline
						54 & 13 & 4 & 1 & 6 & \\\hline
						\multirow[t]{6}{*}{55} & 14 & 3 & 1 & 5 & \\
						& 24 & 8 & 2 & 5 & \\
						& 30 & 17 & 5 & 11 & \\
						& \multirow[t]{2}{*}{\textbf{34}} & \textbf{18} & \textbf{3} & \textbf{5} & \textbf{0, Section \ref{sec:neu:nonexistence}}\\
						& & 21 & 6 & 11 & \\
						& 36 & 23 & 6 & 10 & \\\hline
						\multirow[t]{3}{*}{56} & 27 & 12 & 3 & 7 & \\
						& 30 & 14 & 3 & 6 & \\
						& 33 & 20 & 6 & 12 & \\\hline
						\multirow[t]{3}{*}{57} & 24 & 11 & 3 & 9 & \\
						& 40 & 27 & 6 & 9 & \\
						& 42 & 31 & 10 & 15 & \\\hline
						\multirow[t]{4}{*}{60} & 14 & 4 & 1 & 6 & $\ge 1$, \cite{greaves2019another}\\
						& 17 & 2 & 1 & 4 & \\
						& 35 & 22 & 7 & 15 & \\
						& 38 & 25 & 8 & 15 & \\\hline
						\multirow[t]{5}{*}{63} & 14 & 5 & 1 & 7 & \\
						& 30 & 13 & 3 & 7 & \\
						& 32 & 16 & 4 & 9 & \\
						& \multirow[t]{2}{*}{38} & 21 & 4 & 7 & \\
						&& 22 & 5 & 9 & \\\hline
						\multirow[t]{11}{*}{64} & 18 & 2 & 1 & 4 & \\
						& 21 & 8 & 2 & 8 & \\
						& 28 & 12 & 3 & 8 & $\ge 18$, \cite{goryainovcommunication}\\
						& \multirow[t]{2}{*}{33} & 12 & 2 & 4 & \\
						& & 20 & 6 & 16 & \\
						& 35 & 18 & 4 & 8 & $\ge 1380$, \cite{evans2018smallest,goryainovcommunication}\\
						& 36 & 20 & 5 & 10 & \\
						& 42 & 26 & 5 & 8 & $\ge 1$, \cite{goryainovcommunication}\\
						& 45 & 32 & 10 & 16 & \\
						& 48 & 36 & 11 & 16 & \\
						& 49 & 36 & 6 & 8 & 
					\end{tabular}
				\end{minipage} 
				\caption{Admissible parameters for strictly Neumaier graphs up to $64$ vertices. The last column indicates the number of such strictly Neumaier graphs, where~$\ge x$ means that~$x$ non-isomorphic examples are known to exist, but there may be more. New results from this article are highlighted in bold.}\label{table:neu:parametersstrictlyNG}
			\end{table}

	\bigskip
	
	\section{A proof for the uniqueness of the smallest strictly Neumaier graph}
	\label{sec:neu:uniqueNG16}
	In~\cite{evans2018smallest}, Evans, Goryainov and Panasenko computationally showed that there exists a unique strictly Neumaier graph with parameters~$(16,9,4;2,4)$, the smallest admissible order of a strictly Neumaier graph. This was done by computer search, starting out with a clique and constructing all graphs such that this clique is regular. Below, we instead provide a theoretical (computer-free) uniqueness proof for this parameter set. The benefit such a proof is that it gives us a better understanding of the local structure of the Neumaier graph and is more easily verifiable than a computer proof. Rather than starting from a clique, our proof explores a different approach, starting from a neighborhood of a vertex.
 \medskip
	
	A 9-regular graph~$\G$ on sixteen vertices has diameter two, since two non-adjacent vertices have at least four common neighbors. Therefore, we will often refer to the non-neighborhood of a vertex~$u\in V(\Gamma)$ as~$\Gamma_2(u)$. Figure~\ref{fig:neu:specialgraphs} shows three graphs that will play a role in our proofs: the diamond, paw and bowtie graph.
	
	\begin{figure}[!htb]
		\centering
		\begin{subfigure}[b]{0.19\textwidth}
			\centering
			\begin{adjustbox}{width=0.4\textwidth}
				\begin{tikzpicture}[scale = 1.5,nodes={draw,circle,fill}]
					\node (a) at (0,0) {};
					\node (b) at (0.5,-1) {};
					\node (c) at (-0.5,-1) {};
					\node (d) at (0,-2) {};
					\draw (b) -- (c) -- (a) -- (b) -- (d) -- (c);
				\end{tikzpicture}
			\end{adjustbox}
			\caption{}
			\label{fig:neu:diamondgraph}
		\end{subfigure}
		\begin{subfigure}[b]{0.19\textwidth}
			\centering
			\begin{adjustbox}{width=0.4\textwidth}
				\begin{tikzpicture}[scale = 1.5,nodes={draw,circle,fill}]
					\node (a) at (0,0) {};
					\node (b) at (0.5,-1) {};
					\node (c) at (-0.5,-1) {};
					\node (d) at (0,1) {};
					\draw (a) -- (b) -- (c) -- (a) -- (d);
				\end{tikzpicture}
			\end{adjustbox}
			\caption{}
			\label{fig:neu:pawgraph}
		\end{subfigure}
		\begin{subfigure}[b]{0.19\textwidth}
			\centering
			\begin{adjustbox}{width=0.4\textwidth}
				\begin{tikzpicture}[scale = 1.5,nodes={draw,circle,fill}]
					\node (a) at (0,0) {};
					\node (b) at (0.5,-1) {};
					\node (c) at (-0.5,-1) {};
					\node (d) at (0.5,1) {};
					\node (e) at (-0.5,1) {};
					\draw (a) -- (b) -- (c) -- (a) -- (e) -- (d) -- (a);
				\end{tikzpicture}
			\end{adjustbox}
			\caption{}
			\label{fig:neu:bowtiegraph}
		\end{subfigure}
		\caption{The diamond graph, paw graph and bowtie graph.}
		\label{fig:neu:specialgraphs}
	\end{figure}
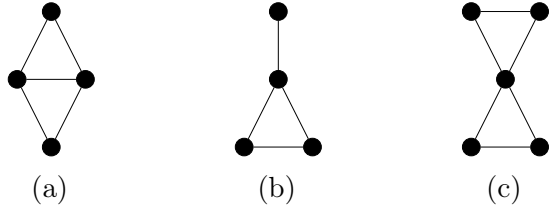

	\begin{lemma} A 4-regular graph on nine vertices which does not have the the diamond graph or~$K_4$ as a subgraph, either admits a partition into triangles or is isomorphic to the graph $\Gamma^{(1)}$ given in Figure~\ref{fig:neu:neighborhood1}\subref{fig:neu:neighborhoodcase1sub1}.
		\label{lem:neu:ninefourgraphs}
	\end{lemma}
	\begin{proof}
		Note that a 4-regular graph~$\G$ on nine vertices has eighteen edges and diameter two. Therefore, we will often refer to the non-neighborhood of a vertex as its set of vertices at distance 2. We will prove the claim by construction, starting with the graph~$K_{1,4}\cup 4K_1$ in Figure~\ref{fig:neu:neighborhoodbase}, which also contains the vertex labeling we will use throughout this proof. We will distinguish three cases, based on the number of edges in~$\Gamma(x)$. The construction steps of each case are illustrated in Figure~\ref{fig:neu:neighborhood1},~\ref{fig:neu:neighborhood2} and~\ref{fig:neu:neighborhood3} respectively, where each subfigure is labeled according with its written explanation.
		
		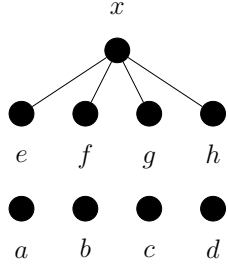
\begin{figure}[!htb]
			\centering
			\begin{adjustbox}{width=0.22\textwidth}
				\begin{tikzpicture}[nodes={draw, circle,fill}]
					\node[label=below:\strut$a$] (a) at (0,0.5){};
					\node[label=below:\strut$b$] (b) at (1,0.5){};
					\node[label=below:\strut$c$] (c) at (2,0.5){};
					\node[label=below:\strut$d$] (d) at (3,0.5){};
					\node[label=below:\strut$e$] (e) at (0,2){};
					\node[label=below:\strut$f$] (f) at (1,2){};
					\node[label=below:\strut$g$] (g) at (2,2){};
					\node[label=below:\strut$h$] (h) at (3,2){};
					\node[label=above:\strut$x$] (x) at (1.5,3){};
					
					\draw (f) -- (x) -- (e); 
					\draw (g) -- (x) -- (h);
					
				\end{tikzpicture}
			\end{adjustbox}
			\caption{The graph~$K_{1,4}\cup 4K_1$ and its vertex labeling.}
			\label{fig:neu:neighborhoodbase}
		\end{figure}
		
		\begin{case}{1}
			$\Gamma(x)$ is an independent set.
		\end{case}
		\noindent Since every vertex has degree four, there must be~$12$ edges between~$\Gamma(x)$ and~$\Gamma_2(x)$, hence~$\Gamma_2(x)$ must contain two edges. These two edges can either form a path or be disjoint. 
		\begin{case}{1.1}
			The edges in~$\Gamma_2(x)$ form a path.
		\end{case}
		\noindent Assume this path to be~$(a,b,c)$. 
	       All vertices in~$\Gamma(x)$ must be adjacent to~$d$ to ensure it has degree four. Up to this point, vertices~$e,f,g,h$ have identical neighborhoods. Each of them must have two more adjacencies with the set~$\{a,b,c\}$. Note that the endpoints of edges~$\{a,b\}$ and~$\{b,c\}$ may not have two common neighbors in~$\Gamma(x)$, because this induces a diamond graph. Therefore, we may without loss of generality assume that~$e$, $f$, $g$ and~$h$ are adjacent to~$\{a,b\}$,~$\{b,c\}$,~$\{a,c\}$ and~$\{a,c\}$, respectively. We will refer to the resulting graph, shown in Figure~\ref{fig:neu:neighborhood1}\subref{fig:neu:neighborhoodcase1sub1}, as~$\Gamma^{(1)}$. 
        
		\begin{case}{1.2}
			The edges in~$\Gamma_2(x)$ are disjoint.
		\end{case}
		\noindent Each of the vertices~$e$,~$f$,~$g$ and~$h$ is adjacent to two vertices in~$\Gamma_2(x)$ which span an edge. Since~$\Gamma_2(x)$ contains only two edges, there must be two vertices in~$\Gamma(x)$ adjacent to the endpoints of the same edge. This forms a diamond graph, contradicting the assumption.
		
		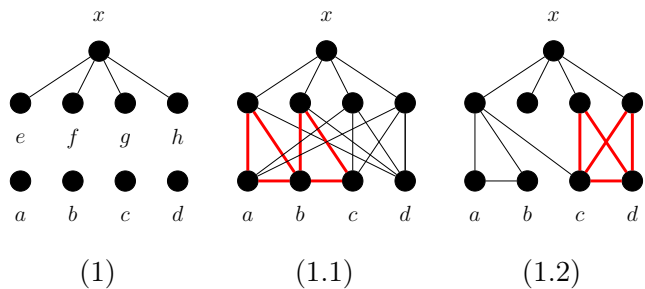
\begin{figure}[!htb]
			\centering
			\renewcommand{\thesubfigure}{(1)}
			\begin{subfigure}[b]{0.19\textwidth}
				\centering
				\begin{adjustbox}{width=0.95\textwidth}
					\begin{tikzpicture}[nodes={draw, circle,fill}]
						\node[label=below:\strut$a$] (a) at (0,0.5){};
						\node[label=below:\strut$b$] (b) at (1,0.5){};
						\node[label=below:\strut$c$] (c) at (2,0.5){};
						\node[label=below:\strut$d$] (d) at (3,0.5){};
						\node[label=below:\strut$e$] (e) at (0,2){};
						\node[label=below:\strut$f$] (f) at (1,2){};
						\node[label=below:\strut$g$] (g) at (2,2){};
						\node[label=below:\strut$h$] (h) at (3,2){};
						\node[label=above:\strut$x$] (x) at (1.5,3){};
						
						\draw (f) -- (x) -- (e); 
						\draw (g) -- (x) -- (h);
						
					\end{tikzpicture}
				\end{adjustbox}
				\caption{}
				\label{fig:neu:neighborhoodcase1}
			\end{subfigure}
			\renewcommand{\thesubfigure}{(1.1)}
			\begin{subfigure}[b]{0.19\textwidth}
				\centering
				\begin{adjustbox}{width=0.95\textwidth}
					\begin{tikzpicture}[nodes={draw, circle,fill}]
						\node[label=below:\strut$a$] (a) at (0,0.5){};
						\node[label=below:\strut$b$] (b) at (1,0.5){};
						\node[label=below:\strut$c$] (c) at (2,0.5){};
						\node[label=below:\strut$d$] (d) at (3,0.5){};
						\node (e) at (0,2){};
						\node (f) at (1,2){};
						\node (g) at (2,2){};
						\node (h) at (3,2){};
						\node[label=above:\strut$x$] (x) at (1.5,3){};
						
						\draw (f) -- (x) -- (e); 
						\draw (g) -- (x) -- (h);
						\draw (a) -- (b) -- (c);
						\draw (e) -- (d) -- (f);
						\draw (g) -- (d) -- (h);
						
						\draw[ultra thick, red] (a) -- (e) -- (b) -- (a);
						\draw[ultra thick, red] (b) -- (c) -- (f) -- (b);
						\draw (c) -- (g) -- (a);
						\draw (a) -- (h) -- (d);
						\draw (h) -- (c);
						
					\end{tikzpicture}
				\end{adjustbox}
				\caption{}
				\label{fig:neu:neighborhoodcase1sub1}
			\end{subfigure}
			\renewcommand{\thesubfigure}{(1.2)}
			\begin{subfigure}[b]{0.19\textwidth}
				\centering
				\begin{adjustbox}{width=0.95\textwidth}
					\begin{tikzpicture}[nodes={draw, circle,fill}]
						\node[label=below:\strut$a$] (a) at (0,0.5){};
						\node[label=below:\strut$b$] (b) at (1,0.5){};
						\node[label=below:\strut$c$] (c) at (2,0.5){};
						\node[label=below:\strut$d$] (d) at (3,0.5){};
						\node (e) at (0,2){};
						\node (f) at (1,2){};
						\node (g) at (2,2){};
						\node (h) at (3,2){};
						\node[label=above:\strut$x$] (x) at (1.5,3){};
						
						\draw (f) -- (x) -- (e); 
						\draw (g) -- (x) -- (h);

						\draw (e) -- (a) -- (b) -- (e) -- (c); 
						\draw[ultra thick, red] (c) -- (d) -- (g) -- (c);
						\draw[ultra thick, red] (c) -- (h) -- (d);
						
					\end{tikzpicture}
				\end{adjustbox}
				\caption{}
				\label{fig:neu:neighborhoodcase1sub2}
			\end{subfigure}
			\caption{Construction of all diamond-free 4-regular graphs on 9 vertices if~$\Gamma(x)$ is an independent set.}
			\label{fig:neu:neighborhood1}
		\end{figure}
		
		\begin{case}{2}
			$\Gamma(x)$ contains a single edge.
		\end{case}
		\noindent Without loss of generality, we assume vertices~$e$ and~$f$ to be adjacent, as shown in Figure~\ref{fig:neu:neighborhood2}. In this case, there will be ten edges between~$\Gamma(x)$ and~$\Gamma_2(x)$. Therefore,~$\Gamma_2(x)$ has three edges, which form either a triangle, a star (i.e.~$K_{1,3}$) or a path. 
		
		\begin{case}{2.1}
			The edges in~$\Gamma_2(x)$ form a triangle.
		\end{case}
		\noindent Without loss of generality, let~$a,b,c$ be the vertices of this triangle. All vertices in~$\Gamma(x)$ must be adjacent to~$d$ to ensure it has degree four. This induces a diamond graph on~$x$,~$e$,~$f$ and~$d$, contradicting the assumption.
		
		\begin{case}{2.2}
			The edges in~$\Gamma_2(x)$ form a star.
		\end{case}
		\noindent Assume~$b$ to be the center vertex of the star. The neighborhoods of~$g$ and~$h$ may not share an edge of the star, as this would induce a diamond graph. Therefore, at least one of them, without loss of generality~$h$, must be adjacent to~$a$,~$c$ and~$d$. Note that the neighborhood of~$h$ is then an independent set, so this graph is isomorphic to~$\Gamma^{(1)}$.
		
		\begin{case}{2.3}
			The edges in~$\Gamma_2(x)$ form a path.
		\end{case}
		\noindent Without loss of generality, let the path be~$\{a,b,c,d\}$. 
		Vertex~$h$ cannot be adjacent to three consecutive path vertices, because this induces a diamond graph. Without loss of generality, it is therefore adjacent to~$a$,~$c$ and~$d$. Similarly, vertex~$g$ cannot be adjacent to three consecutive path vertices, and additionally may not share an edge with~$\Gamma(h)$, so~$g$ is adjacent to~$a$,~$b$ and~$d$. This give rise to three disjoint triangles,~$\{x,e,f\}$,~$\{a,b,g\}$ and~$\{c,d,h\}$, so if this construction leads to a diamond-free graph, it must have a partition into triangles. 
		
		\begin{figure}[!htb]
			\centering
			\renewcommand{\thesubfigure}{(2)}
			\begin{subfigure}[b]{0.19\textwidth}
				\centering
				\begin{adjustbox}{width=0.95\textwidth}
					\begin{tikzpicture}[nodes={draw, circle,fill}]
						\node[label=below:\strut$a$] (a) at (0,0.5){};
						\node[label=below:\strut$b$] (b) at (1,0.5){};
						\node[label=below:\strut$c$] (c) at (2,0.5){};
						\node[label=below:\strut$d$] (d) at (3,0.5){};
						\node[label=below:\strut$e$] (e) at (0,2){};
						\node[label=below:\strut$f$] (f) at (1,2){};
						\node[label=below:\strut$g$] (g) at (2,2){};
						\node[label=below:\strut$h$] (h) at (3,2){};
						\node[label=above:\strut$x$] (x) at (1.5,3){};
						
						\draw (f) -- (x) -- (e) -- (f); 
						\draw (g) -- (x) -- (h);
						
					\end{tikzpicture}
				\end{adjustbox}
				\caption{}
				\label{fig:neu:neighborhoodcase2}
			\end{subfigure}
			\renewcommand{\thesubfigure}{(2.1)}
			\begin{subfigure}[b]{0.19\textwidth}
				\centering
				\begin{adjustbox}{width=0.95\textwidth}
					\begin{tikzpicture}[nodes={draw, circle,fill}]
						\node[label=below:\strut$a$] (a) at (0,0.5){};
						\node[label=below:\strut$b$] (b) at (1,0.5){};
						\node[label=below:\strut$c$] (c) at (2,0.5){};
						\node[label=below:\strut$d$] (d) at (3,0.5){};
						\node (e) at (0,2){};
						\node (f) at (1,2){};
						\node (g) at (2,2){};
						\node (h) at (3,2){};
						\node[label=above:\strut$x$] (x) at (1.5,3){};
						
						\draw (f) -- (x) -- (e) -- (f); 
						\draw (g) -- (x) -- (h);
						
						\draw[ultra thick, red] (x) -- (e) -- (f) -- (x);
						\draw[ultra thick, red] (e) -- (d) -- (f);
						\draw (a) -- (b) -- (c);
						\draw (a) to [out=335,in=205] (c);
						\draw (g) -- (d) -- (h);
						
					\end{tikzpicture}
				\end{adjustbox}
				\caption{}
				\label{fig:neu:neighborhoodcase2sub1a}
			\end{subfigure}
			\renewcommand{\thesubfigure}{(2.2)}
			\begin{subfigure}[b]{0.19\textwidth}
				\centering
				\begin{adjustbox}{width=0.95\textwidth}
					\begin{tikzpicture}[nodes={draw, circle,fill}]
						\node[label=below:\strut$a$] (a) at (0,0.5){};
						\node[label=below:\strut$b$] (b) at (1,0.5){};
						\node[label=below:\strut$c$] (c) at (2,0.5){};
						\node[label=below:\strut$d$] (d) at (3,0.5){};
						\node (e) at (0,2){};
						\node (f) at (1,2){};
						\node (g) at (2,2){};
						\node (h) at (3,2){};
						\node[label=above:\strut$x$] (x) at (1.5,3){};
						\node[fill = none, draw=red, ultra thick, dashed, inner sep = 8pt] (y) at (3,2){};
						
						\draw (f) -- (x) -- (e) -- (f); 
						\draw (g) -- (x) -- (h);
						
						\draw (a) -- (b) -- (c);
						\draw (b) to [out=340,in=200] (d);
						\draw (a) -- (h) -- (c);
						\draw (h) -- (d);
						
					\end{tikzpicture}
				\end{adjustbox}
				\caption{}
				\label{fig:neu:neighborhoodcase2sub2a}
			\end{subfigure}
			\renewcommand{\thesubfigure}{(2.3)}
			\begin{subfigure}[b]{0.19\textwidth}
				\centering
				\begin{adjustbox}{width=0.95\textwidth}
					\begin{tikzpicture}[nodes={draw, circle,fill}]
						\node[label=below:\strut$a$] (a) at (0,0.5){};
						\node[label=below:\strut$b$] (b) at (1,0.5){};
						\node[label=below:\strut$c$] (c) at (2,0.5){};
						\node[label=below:\strut$d$] (d) at (3,0.5){};
						\node (e) at (0,2){};
						\node (f) at (1,2){};
						\node (g) at (2,2){};
						\node (h) at (3,2){};
						\node[label=above:\strut$x$] (x) at (1.5,3){};
						
						\draw (f) -- (x) -- (e) -- (f); 
						\draw (g) -- (x) -- (h);
						
						\draw (a) -- (b) -- (c) -- (d);
						\draw (a) -- (h) -- (c);
						\draw (h) -- (d);
						
						\draw[ultra thick, red] (x) -- (e) -- (f) -- (x);
						\draw[ultra thick, red] (a) -- (b) -- (g) -- (a);
						\draw[ultra thick, red] (d) -- (h) -- (c) -- (d);
						\draw (g) -- (d);
						
					\end{tikzpicture}
				\end{adjustbox}
				\caption{}
				\label{fig:neu:neighborhoodcase2sub3}
			\end{subfigure}
			\caption{Construction of all diamond-free 4-regular graphs on 9 vertices if~$\Gamma(x)$ contains one edge.}
			\label{fig:neu:neighborhood2}
		\end{figure}
		
		\begin{case}{3}
			$\Gamma(x)$ contains at least two edges.
		\end{case}
		\noindent Note that~$\Gamma(x)$ contains exactly two edges, otherwise $\{x\}\cup\Gamma(x)$ (and thus also $\Gamma$) has the diamond graph or~$K_4$ as a subgraph. These two edges must be disjoint, otherwise $\Gamma$ also admits a diamond graph as subgraph. Without loss of generality, we can assume that~$e\sim f$ and~$g\sim h$. Then there are eight edges between~$\Gamma(x)$ and~$\Gamma_2(x)$, so the induced graph on~$\Gamma_2(x)$ has four edges. These edges either form a 4-cycle or the paw graph from Figure~\ref{fig:neu:pawgraph}.
		
		\begin{case}{3.1}
			The edges in~$\Gamma_2(x)$ form a paw graph. 
		\end{case}
		\noindent Let~$d$ be the vertex of degree one (see Figure~\ref{fig:neu:neighborhood3}\subref{fig:neu:neighborhoodcase3sub1a}). Vertex~$d$ has three more neighbors, so its neighborhood includes either~$\{e,f\}$ or~$\{g,h\}$. This creates a diamond graph, contradicting the assumption.
		
		\begin{case}{3.2}
			The edges in~$\Gamma_2(x)$ form a 4-cycle.
		\end{case}
		\noindent All vertices of~$\Gamma(x)$ have two neighbors in~$\Gamma_2(x)$. These two neighbors are either adjacent in the 4-cycle or not. If they are adjacent, no other vertex in~$\Gamma(x)$ can have the same two neighbors. Moreover,~$e$ and~$f$ (and similarly~$g$ and~$h$) have no common neighbors in~$\Gamma_2(x)$, because this would also induce a diamond graph, contradicting the assumption. So, either~$e$ and~$f$ (and similarly~$g$ and~$h$) are both adjacent to consecutive cycle vertices or both are not. The following two cases then cover all options up to symmetry.
		
		\begin{dense_enumerate}{label=(\alph*)}
			\item Suppose without loss of generality that~$e$ is adjacent to two consecutive cycle vertices. Then so is~$f$, which leads to the graph in Figure~\ref{fig:neu:neighborhood3}\subref{fig:neu:neighborhoodcase3sub2a}. This graph admits a  partition into triangles.
			\item If~$e$,~$f$,~$g$ and~$h$ each have two neighbors in the cycle that are not adjacent, we get the graph in Figure~\ref{fig:neu:neighborhood3}\subref{fig:neu:neighborhoodcase3sub2b}. The neighborhood of~$a$ is an independent set, so this graph must be isomorphic to~$\Gamma^{(1)}$.\qedhere
		\end{dense_enumerate}
	\end{proof}
	
	\begin{figure}[!htb]
		\centering
		\renewcommand{\thesubfigure}{(3)}
		\begin{subfigure}[b]{0.19\textwidth}
			\centering
			\begin{adjustbox}{width=0.95\textwidth}
				\begin{tikzpicture}[nodes={draw, circle,fill}]
					\node[label=below:\strut$a$] (a) at (0,0.5){};
					\node[label=below:\strut$b$] (b) at (1,0.5){};
					\node[label=below:\strut$c$] (c) at (2,0.5){};
					\node[label=below:\strut$d$] (d) at (3,0.5){};
					\node[label=below:\strut$e$] (e) at (0,2){};
					\node[label=below:\strut$f$] (f) at (1,2){};
					\node[label=below:\strut$g$] (g) at (2,2){};
					\node[label=below:\strut$h$] (h) at (3,2){};
					\node[label=above:\strut$x$] (x) at (1.5,3){};
					
					\draw (f) -- (x) -- (e) -- (f); 
					\draw (g) -- (x) -- (h) -- (g);
					
				\end{tikzpicture}
			\end{adjustbox}
			\caption{}
			\label{fig:neu:neighborhoodcase3}
		\end{subfigure}
		\renewcommand{\thesubfigure}{(3.1)}
		\begin{subfigure}[b]{0.19\textwidth}
			\centering
			\begin{adjustbox}{width=0.95\textwidth}
				\begin{tikzpicture}[nodes={draw, circle,fill}]
					\node[label=below:\strut$a$] (a) at (0,0.5){};
					\node[label=below:\strut$b$] (b) at (1,0.5){};
					\node[label=below:\strut$c$] (c) at (2,0.5){};
					\node[label=below:\strut$d$] (d) at (3,0.5){};
					\node (e) at (0,2){};
					\node (f) at (1,2){};
					\node (g) at (2,2){};
					\node (h) at (3,2){};
					\node[label=above:\strut$x$] (x) at (1.5,3){};
					
					\draw (f) -- (x) -- (e) -- (f); 
					\draw (g) -- (x) -- (h) -- (g);
					
					\draw (a) to [out=335,in=205] (c);
					\draw (a) -- (b) -- (c) -- (d);
					\draw (f) -- (d);
					\draw[ultra thick, red] (x) -- (g) -- (h) -- (x);
					\draw[ultra thick, red] (g) -- (d) -- (h);
					
				\end{tikzpicture}
			\end{adjustbox}
			\caption{}
			\label{fig:neu:neighborhoodcase3sub1a}
		\end{subfigure}
		\renewcommand{\thesubfigure}{(3.2a)}
		\begin{subfigure}[b]{0.19\textwidth}
			\centering
			\begin{adjustbox}{width=0.95\textwidth}
				\begin{tikzpicture}[nodes={draw, circle,fill}]
					\node[label=below:\strut$a$] (a) at (0,0.5){};
					\node[label=below:\strut$b$] (b) at (1,0.5){};
					\node[label=below:\strut$c$] (c) at (2,0.5){};
					\node[label=below:\strut$d$] (d) at (3,0.5){};
					\node (e) at (0,2){};
					\node (f) at (1,2){};
					\node (g) at (2,2){};
					\node (h) at (3,2){};
					\node[label=above:\strut$x$] (x) at (1.5,3){};
					
					\draw[] (f) -- (x) -- (e) -- (f); 
					\draw[ultra thick, red] (g) -- (x) -- (h) -- (g);
					
					\draw[] (a) to [out=340,in=200] (d);
					\draw[] (b) -- (c);
					\draw[ultra thick, red] (a) -- (b);
					\draw[ultra thick, red] (c) -- (d);
					\draw[ultra thick, red] (a) -- (e) -- (b);
					\draw[ultra thick, red] (c) -- (f) -- (d);
					
				\end{tikzpicture}
			\end{adjustbox}
			\caption{}
			\label{fig:neu:neighborhoodcase3sub2a}
		\end{subfigure}
		\renewcommand{\thesubfigure}{(3.2b)}
		\begin{subfigure}[b]{0.19\textwidth}
			\centering
			\begin{adjustbox}{width=0.95\textwidth}
				\begin{tikzpicture}[nodes={draw, circle,fill}]
					\node[label=below:\strut$a$] (a) at (0,0.5){};
					\node[label=below:\strut$b$] (b) at (1,0.5){};
					\node[label=below:\strut$c$] (c) at (2,0.5){};
					\node[label=below:\strut$d$] (d) at (3,0.5){};
					\node (e) at (0,2){};
					\node (f) at (1,2){};
					\node (g) at (2,2){};
					\node (h) at (3,2){};
					\node[label=above:\strut$x$] (x) at (1.5,3){};
					\node[fill = none, draw=red, ultra thick, dashed, inner sep = 8pt] (y) at (0,0.5){};
					
					\draw[ultra thick, red] (f) -- (x) -- (e) -- (f); 
					\draw[ultra thick, red] (g) -- (x) -- (h) -- (g);
					
					\draw (a) to [out=340,in=200] (d);
					\draw (a) -- (b) -- (c) -- (d);
					\draw (a) -- (e) -- (c);
					\draw (a) -- (g) -- (c);
					\draw (b) -- (f) -- (d);
					\draw (b) -- (h) -- (d);
					
				\end{tikzpicture}
			\end{adjustbox}
			\caption{}
			\label{fig:neu:neighborhoodcase3sub2b}
		\end{subfigure}
		\caption{Construction of all diamond-free 4-regular graphs on 9 vertices if~$\Gamma(x)$ contains two edges.}
		\label{fig:neu:neighborhood3}
	\end{figure}
	
	\begin{lemma} 
		If~$\Gamma$ is a strictly Neumaier graph with parameters~$(16,9,4;2,4)$, then~$\Gamma$ has a vertex~$u\in V(\Gamma)$ such that~$\Gamma(v)\simeq \Gamma^{(1)}$.
		\label{lem:neu:gamma1}
	\end{lemma}
	\begin{proof}
		Note that the neighborhood of a vertex in~$\Gamma$ may not contain the diamond graph as a subgraph, because this induces a pair of cliques which intersect in three vertices, violating the 2-regularity of the cliques and thus contradicting Theorem~\ref{th:maxcliques}$\ref{it:allCreg}$. It also does not contain a~$K_4$, as this would violate Theorem~\ref{th:maxcliques}(i). Suppose that no~$u\in V(\Gamma)$ has a neighborhood isomorphic to~$\Gamma^{(1)}$. Then by Lemma~\ref{lem:neu:ninefourgraphs} the neighborhood of any vertex~$u\in V(\Gamma)$ contains a partition into triangles, so~$\Gamma$ contains three cliques which intersect only in~$u$. Because of the 2-regularity, every vertex in~$\Gamma_2(u)$ must have two neighbors in each of these cliques, and hence it has exactly six neighbors in~$\Gamma(u)$. However, this means that~$\Gamma$ is~$6$-co-edge-regular, and hence not strictly Neumaier.
	\end{proof}
	
	\begin{lemma} \leavevmode
		\begin{dense_enumerate}{label=(\roman*),align=left,leftmargin=*}
			\item There exists a unique graph on six vertices with degree sequence~$(1,3,3,3,3,5)$.\label{lem:neu:nonexistence1}
			\item There are no graphs on six vertices with degree sequence~$(1,1,4,4,4,4)$, $(1,2,2,\allowbreak3,5,5)$ or~$(1,2,2,4,4,5)$. \label{lem:neu:nonexistence2}
			\item If a graph has degree sequence~$(1,2,3,3,4,5)$, it must have~$K_4$ as a subgraph.\label{lem:neu:nonexistence3}
		\end{dense_enumerate}
		\label{lem:neu:nonexistence}
	\end{lemma}
	\begin{proof}\leavevmode
		\begin{enumerate}[label=(\roman*),align=left,leftmargin=*]
			\item This graph can be constructed uniquely in the following way. Begin with a star graph on 6 vertices. The center vertex has degree five, the others degree one. This means there are four vertices which have two additional neighbors. The only way to achieve this is by adding a 4-cycle. 
			\item Suppose there exists a graph with degree sequence~$(1,1,4,4,4,4)$. The two vertices of degree one each have one (possibly the same) neighbor of degree four. Then there exists a vertex of degree four which is not adjacent to either, a contradiction.
			
			Suppose~$\Gamma$ is a graph on 6 vertices with degree sequence~$(1,2,2,3,5,5)$. The two vertices of degree five must both be adjacent to all others. However, this is not possible, because~$\Gamma$ has a vertex of degree one.
			
			If a graph has degree sequence~$(1,2,2,4,4,5)$, its two vertices of degree four have (at least) two common neighbors by the pigeonhole principle. One of these must be the vertex of degree five, the other we will call~$w$. Note that the vertex of degree five is adjacent to all others, so~$w$ has degree at least three. This contradicts the degree sequence, as all other vertices should have degree at most two.
			\item Suppose~$\Gamma$ is a graph with the given degree sequence and let~$u$ and~$w$ denote the vertices of degree four and five. Then necessarily~$w\sim u$ and~$|\Gamma(w)\cap \Gamma(u)| = 3$. The remaining neighbor of~$w$ is the vertex of degree one. Note that~$\Gamma$ has nine edges in total, so there must be one additional edge between the vertices of~$\Gamma(u)\cap \Gamma(w)$, creating a~$K_4$.\qedhere
		\end{enumerate}
	\end{proof}
	
	\begin{theorem}\label{th:neu:ng16}
		Up to isomorphism, there exists a unique strictly Neumaier graph with parameter set~$(16,9,4;2,4)$.
	\end{theorem}
	\begin{proof}
		Suppose that~$\Gamma$ is a strictly Neumaier graph with parameters~$(16,9,4;2,4)$. Lemma~\ref{lem:neu:gamma1} implies that there must be a vertex~$u\in V(\Gamma)$ such that~$\Gamma(u) \simeq \Gamma^{(1)}$. We will use the vertex labeling of Figure~\ref{fig:neu:gamma(u)} to refer to the vertices of~$\Gamma(u)$. Together with~$u$, the bowtie subgraph on~$B:=(\Gamma(x)\cap \Gamma(u))\cup \{x\}$ forms two 4-cliques, which intersect in~$u$ and~$x$. Within~$\Gamma(u)$, every vertex has two neighbors in each of these cliques if it is not contained in it, so all vertices of this subgraph comply with the 2-regularity requirement. This must also be the case for all vertices in~$\Gamma_2(u)$. We will show that this uniquely determines the edges between the~$B$ and~$\Gamma_2(u)$, up to isomorphism. 
		
		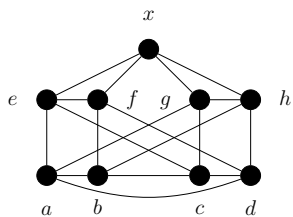
\begin{figure}[!htb]
			\centering
			\begin{adjustbox}{width=0.28\textwidth}
				\begin{tikzpicture}[nodes={draw, circle,fill}]
					\node[label=below:\strut$a$] (a) at (-1,0.5){};
					\node[label=below:\strut$b$] (b) at (0,0.5){};
					\node[label=below:\strut$c$] (c) at (2,0.5){};
					\node[label=below:\strut$d$] (d) at (3,0.5){};
					\node[label= left:\strut$e$] (e) at (-1,2){};
					\node[label= right:\strut$f$] (f) at (0,2){};
					\node[label= left:\strut$g$] (g) at (2,2){};
					\node[label= right:\strut$h$] (h) at (3,2){};
					\node[label=above:\strut$x$] (x) at (1,3){};
					
					\draw (f) -- (x) -- (e) -- (f); 
					\draw (g) -- (x) -- (h) -- (g);
					
					\draw (a) to [out=340,in=200] (d);
					\draw (a) -- (b) -- (c) -- (d);
					\draw (a) -- (e) -- (c);
					\draw (a) -- (g) -- (c);
					\draw (b) -- (f) -- (d);
					\draw (b) -- (h) -- (d);
					
				\end{tikzpicture}
			\end{adjustbox}
			\caption{The neighborhood of vertex~$u$.}
			\label{fig:neu:gamma(u)}
		\end{figure}
		
		Consider the subgraph spanned by~$B\cup\Gamma_2(u)$. The vertices of~$B$ must have four more neighbors in~$\Gamma_2(u)$ each, and every vertex of~$\Gamma_2(u)$ must have two neighbors in the cliques~$\{u, x, e, f\}$ and~$\{u, x, g, h\}$ because of 2-regularity. Therefore, two vertices of~$\Gamma_2(u)$ are adjacent to~$e$,~$f$,~$g$ and~$h$, while the others are adjacent to~$x$ and one vertex from each of the sets~$\{e,f\}$ and~$\{g,h\}$. Up to symmetry, there are two ways to choose these edges. Either two vertices of~$\Gamma_2(u)$ are adjacent to~$e$ and~$g$ and two to~$f$ and~$h$, or the four remaining vertices are adjacent to~$e$ and~$g$,~$g$ and~$f$,~$f$ and~$h$ and~$h$ and~$e$, respectively. Figure~\ref{fig:neu:bowtiefigure} shows both options, as well as the vertex labeling we will assume throughout the rest of the proof for~$\Gamma_2(u)$.
		
		\begin{figure}[!htb]
			\centering
			\begin{subfigure}[b]{0.3\textwidth}
				\centering
				\begin{adjustbox}{width=0.95\textwidth}
					\begin{tikzpicture}[nodes={draw, circle,fill}]
						\node[label=above:\strut\Large$v_1$] (v1) at (-1.5,2.5){};
						\node[label=above:\strut\Large$v_2$] (v2) at (1.5,2.5){};
						\node[label=above:\strut\Large$v_3$] (v3) at (-3,0){};
						\node[label=above:\strut\Large$v_4$] (v4) at (3,0){};
						\node[label=below:\strut\Large$v_5$] (v5) at (-1.5,-2.5){};
						\node[label=below:\strut\Large$v_6$] (v6) at (1.5,-2.5){};
						\node[label= above:\strut\Large$e$] (e) at (-1.5,1){};
						\node[label= below:\strut\Large$f$] (f) at (-1.5,-1){};
						\node[label= above:\strut\Large$g$] (g) at (1.5,1){};
						\node[label= below:\strut\Large$h$] (h) at (1.5,-1){};
						\node[label=above:\strut\Large$x$] (x) at (0,0){};
						
						\draw (f) -- (x) -- (g) -- (h) -- (x) -- (e) -- (f); 
						
					\end{tikzpicture}
				\end{adjustbox}
				\caption{}
				\label{fig:neu:bowtie}
			\end{subfigure}
			\begin{subfigure}[b]{0.32\textwidth}
				\centering
				\begin{adjustbox}{width=0.95\textwidth}
					\begin{tikzpicture}[nodes={draw, circle,fill}]
                        \clip (-4.5,-4.5) rectangle (4.5,4.5);
						\node (v1) at (-1.5,2.5){};
						\node (v2) at (1.5,2.5){};
						\node (v3) at (-3,0){};
						\node (v4) at (3,0){};
						\node (v5) at (-1.5,-2.5){};
						\node (v6) at (1.5,-2.5){};
						\node (e) at (-1.5,1){};
						\node (f) at (-1.5,-1){};
						\node (g) at (1.5,1){};
						\node (h) at (1.5,-1){};
						\node (x) at (0,0){};
						
						\draw (f) -- (x) -- (g) -- (h) -- (x) -- (e) -- (f); 
						\draw (e) -- (v1) -- (x) -- (v2) -- (g);
						\draw (v1) -- (g);
						\draw (v2) -- (e);
						\draw (e) -- (v3) -- (f);
						\draw (g) -- (v4) -- (h);
						\draw (f) -- (v5) -- (x) -- (v6) -- (h);
						\draw (v5) -- (h);
						\draw (v6) -- (f);
						
						\draw (v3) to [out=120,in=50, looseness=2.7] (g);
						\draw (v3) to [out=240,in=-50, looseness=2.7] (h);
						\draw (v4) to [out=60,in=130, looseness=2.5] (e);
						\draw (v4) to [out=-60,in=-130, looseness=2.5] (f);
						
					\end{tikzpicture}
				\end{adjustbox}
				\caption{}
				\label{fig:neu:bowtiecomplete}
			\end{subfigure}
			\begin{subfigure}[b]{0.32\textwidth}
				\centering
				\begin{adjustbox}{width=0.95\textwidth}
					\begin{tikzpicture}[nodes={draw, circle,fill}]
						\clip (-4.5,-4.5) rectangle (4.5,4.5);
                        \node (v1) at (-1.5,2.5){};
						\node (v2) at (1.5,2.5){};
						\node (v3) at (-3,0){};
						\node (v4) at (3,0){};
						\node (v5) at (-1.5,-2.5){};
						\node (v6) at (1.5,-2.5){};
						\node (e) at (-1.5,1){};
						\node (f) at (-1.5,-1){};
						\node (g) at (1.5,1){};
						\node (h) at (1.5,-1){};
						\node (x) at (0,0){};
						
						\draw (f) -- (x) -- (g) -- (h) -- (x) -- (e) -- (f); 
						\draw (e) -- (v1) -- (x) -- (v2) -- (g);
						\draw (v1) -- (g);
						\draw (v2) to [out=130,in=175, looseness=3] (f);
						\draw (e) -- (v3) -- (f);
						\draw (g) -- (v4) -- (h);
						\draw (f) -- (v5) -- (x) -- (v6) -- (h);
						\draw (v5) -- (h);
						\draw (v6) to [out=-130,in=185, looseness=3] (e);
						
						\draw (v3) to [out=120,in=50, looseness=2.7] (g);
						\draw (v3) to [out=240,in=-50, looseness=2.7] (h);
						\draw (v4) to [out=60,in=130, looseness=2.5] (e);
						\draw (v4) to [out=-60,in=-130, looseness=2.5] (f);
					\end{tikzpicture}
				\end{adjustbox}
				\caption{}
				\label{fig:neu:bowtiecomplete2}
			\end{subfigure}
			\caption{The edges between~$B$ and~$\Gamma_2(u)$.}
			\label{fig:neu:bowtiefigure}
		\end{figure}
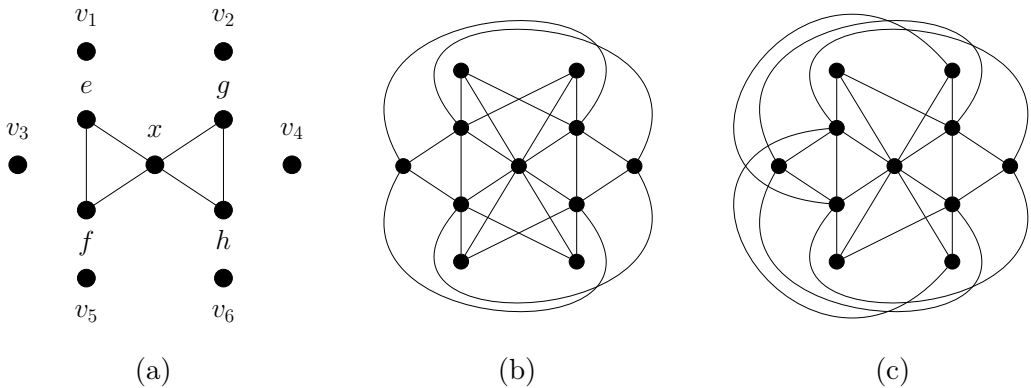
		
		\begin{case}{1} \label{case:case1} The edges between~$B$ and~$\Gamma_2(u)$ are as in Figure~\ref{fig:neu:bowtiecomplete}.
		\end{case}
		\noindent We will now focus on the remaining edges between~$\Gamma(u)$ and~$\Gamma_2(u)$. Vertices~$a$,~$b$,~$c$ and~$d$ must have four neighbors in~$\Gamma_2(u)$ since~$\Gamma$ is 9-regular. This means that~$a$ and~$c$ could have two, three or four common neighbors in~$\Gamma_2(u)$. We branch on each of these cases, denoting~$S:=\Gamma(a)\cap \Gamma(c) \cap\Gamma_2(u)$ for short. Before we do so, we note the following useful fact. 
		
		Both~$a$ and~$c$ are adjacent to~$e$ and~$g$. To satisfy~$4$-edge-regularity, they both must have three common neighbors with~$e$ as well as with~$g$ in~$\Gamma_2(u)$, as they currently only have~$u$ as a common neighbor. Therefore,~$a$ and~$c$ must have three neighbors in the set~$\{v_1, v_2, v_3, v_4\}$ and one in~$\{v_5, v_6\}$. Similarly,~$b$ and~$d$ have one neighbor in~$\{v_1, v_2\}$ and three in~$\{v_3, v_4, v_5, v_6\}$. Since~$a$ and~$c$ are also adjacent to~$b$ and~$d$, both~$a$ and~$c$ must have three common neighbors in~$\Gamma_2(u)$ with~$b$ as well as with~$d$. 
		
		\begin{case}{1.1}
			$|S| = 2$.
		\end{case}
		\noindent By the pigeonhole principle, the two common neighbors lie in~$\{v_1, v_2, v_3, v_4\}$. It also follows from the above observation that they are adjacent to~$b$ and~$d$, otherwise these vertices cannot have three common neighbors with both~$a$ and~$c$. As a result,~$S$ cannot include both~$v_1$ and~$v_2$. We may therefore assume without loss of generality that either~$S = \{v_3, v_4\}$ or~$S=\{v_1,v_3\}$. We show that neither option leads to a valid strictly Neumaier graph.
		
		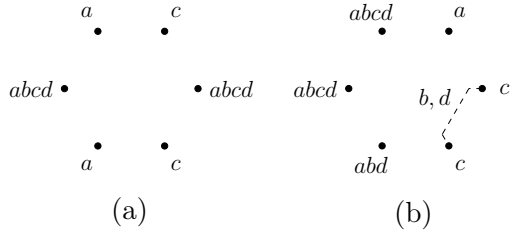
\begin{figure}[!htb]
			\centering
			\begin{subfigure}[c]{0.24\textwidth}
				\centering
				\begin{adjustbox}{width=0.95\textwidth}
					\begin{tikzpicture}[scale=0.4]
						\foreach \x in {0,60,120,180,240,300}{
							\draw[fill=black] (\x:3) circle (4pt);
						}
						\node at (120:4) {$a$};
						\node at (60:4){$c$};
						\node at (180:4.5){$abcd$};
						\node at (0:4.5){$abcd$};
						\node at (240:4){$a$};
						\node at (300:4){$c$};
						\node[draw=none,fill=none] at(0:5.5){};
						\node[draw=none,fill=none] at(180:5.5){};
					\end{tikzpicture}
				\end{adjustbox}
				\caption{}
				\label{fig:neu:scase1sub1a}
			\end{subfigure}
			\begin{subfigure}[c]{0.24\textwidth}
				\centering
				\begin{adjustbox}{width=0.95\textwidth}
					\begin{tikzpicture}[scale=0.4]
						\foreach \x in {0,60,120,180,240,300}{
							\draw[fill=black] (\x:3) circle (4pt);
						}
						\node at (120:4) {$abcd$};
						\node at (60:4){$a$};
						\node at (180:4.5){$abcd$};
						\node at (0:4){$c$};
						\node at (240:4){$abd$};
						\node at (300:4){$c$};
						
						\draw[dashed] (0:3) to (0:2.4);
						\draw[dashed] (300:2.4) to (300:3);
						\draw[dashed] (0:2.4) to (300:2.4);
						\node at (330:1){$b,d$};
						\node[draw=none,fill=none] at(0:5.5){};
						\node[draw=none,fill=none] at(180:5.5){};
					\end{tikzpicture}
				\end{adjustbox}
				\caption{}
				\label{fig:neu:scase1sub1b}
			\end{subfigure}
			\caption{The edges between~$\Gamma_2(u)$ and~$\{a, b, c, d\}$ when~$|S| = 2$.}
			\label{fig:neu:scases1sub1}
		\end{figure}

        \begin{enumerate}
            \item[(a)] If~$S = \{v_3, v_4\}$, we know that~$v_3, v_4 \sim b, d$. As~$a$ and~$c$ have identical neighborhoods so far, we may also assume that~$a\sim v_1, v_5$ and~$c\sim v_2, v_6$, which results in the graph schematically depicted in Figure~\ref{fig:neu:scase1sub1a}. Both~$b$ and~$d$ have one more neighbor in the sets~$\{v_1, v_2\}$ and~$\{v_5, v_6\}$, and one more common neighbor with both~$a$ and~$c$. Either they share the same two neighbors, without loss of generality~$v_1$ and~$v_6$, or they have distinct neighbors, say~$b\sim v_1, v_6$ and~$d\sim v_2, v_5$. In the first case, the subgraph induced on~$\Gamma_2(u)$ has two vertices of degree five,~$v_2$ and~$v_5$, which must be adjacent to all other vertices. However,~$v_3$ and~$v_4$ have degree one in this graph, so this is not possible. The second case implies that the subgraph induced on~$\Gamma_2(u)$ has degree sequence~$(1,1,4,4,4,4)$, contradicting Lemma~\ref{lem:neu:nonexistence}$\ref{lem:neu:nonexistence2}$. 
            \item[(b)] If~$S=\{v_1,v_3\}$, we know that~$v_1, v_3 \sim b, d$ and every vertex in~$\Gamma_2(u)\setminus S$ is adjacent to either~$a$ or~$c$. Without loss of generality, let~$a\sim v_2, v_5$ and~$c\sim v_4, v_6$. Note that~$b,d\not\sim v_2$, because we showed at the start of Case~\ref{case:case1} that they have exactly one neighbor in~$\{v_1, v_2\}$. Then~$b,d\sim v_5$, and they both have one neighbor among~$v_4, v_6$. Figure~\ref{fig:neu:scase1sub1b} shows an overview of these edges. For the induced graph on~$\Gamma_2(u)$, this leads to degree sequence~$(1,2,3,3,4,5)$ or~$(1,2,2,3,5,5)$. By Lemma~\ref{lem:neu:nonexistence}$\ref{lem:neu:nonexistence3}$, the first option is not possible, since it creates a 4-clique in which~$u$ has no neighbors. In the second case, two vertices have degree five in~$\Gamma_2(u)$, so they must both be adjacent to all others. However, this cannot happen, because~$\Gamma_2(u)$ also has a vertex of degree one.
        \end{enumerate}

		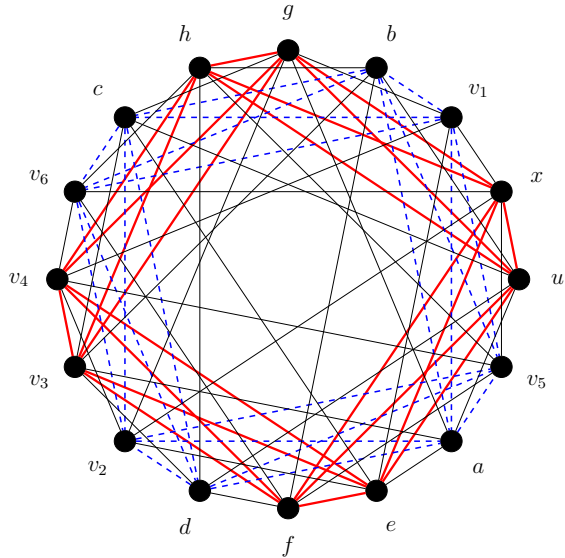
\begin{figure}[!htb]
			\centering
			\begin{adjustbox}{width=0.5\textwidth}
				\begin{tikzpicture}[scale = 1.4,nodes={draw=none, fill=none}]
					
					\node[draw, circle,fill] (u) at (0:3) {};
					\node[draw, circle,fill] (x) at (22.5:3){};
					\node[draw, circle,fill] (v1) at (45:3){};
					\node[draw, circle,fill] (b) at (67.5:3){};
					\node[draw, circle,fill] (g) at (90:3){};
					\node[draw, circle,fill] (h) at (112.5:3){};
					\node[draw, circle,fill] (c) at (135:3) {};
					\node[draw, circle,fill] (v6) at (157.5:3){};
					\node[draw, circle,fill] (v4) at (180:3){};
					\node[draw, circle,fill] (v3) at (202.5:3){};
					\node[draw, circle,fill] (v2) at (225:3){};
					\node[draw, circle,fill] (d) at (247.5:3){};
					\node[draw, circle,fill] (f) at (270:3) {};
					\node[draw, circle,fill] (e) at (292.5:3){};
					\node[draw, circle,fill] (a) at (315:3){};
					\node[draw, circle,fill] (v5) at (337.5:3){};
					
					\node at (0:3.5) {$u$};
					\node at (22.5:3.5){$x$};
					\node at (45:3.5){$v_1$};
					\node at (67.5:3.5){$b$};
					\node at (90:3.5){$g$};
					\node at (112.5:3.5){$h$};
					\node at (135:3.5) {$c$};
					\node at (157.5:3.5){$v_6$};
					\node at (180:3.5){$v_4$};
					\node at (202.5:3.5){$v_3$};
					\node at (225:3.5){$v_2$};
					\node at (247.5:3.5){$d$};
					\node at (270:3.5) {$f$};
					\node at (292.5:3.5){$e$};
					\node at (315:3.5){$a$};
					\node at (337.5:3.5){$v_5$};
					
					\draw[very thick, red] (u) --  (x);
					\draw[very thick, red] (g) --  (x);
					\draw[very thick, red] (u) --  (g);
					\draw[very thick, red] (g) --  (h);
					\draw[very thick, red] (h) --  (x);
					\draw[very thick, red] (u) --  (h);
					
					\draw[very thick, red] (e) --  (f);
					\draw[very thick, red] (v3) --  (f);
					\draw[very thick, red] (e) --  (v3);
					\draw[very thick, red] (v3) --  (v4);
					\draw[very thick, red] (v4) --  (f);
					\draw[very thick, red] (e) --  (v4);

					\draw[very thick, red] (e) --  (x);
					\draw[very thick, red] (u) --  (e);
					\draw[very thick, red] (f) --  (x);
					\draw[very thick, red] (u) --  (f);
					
					\draw[very thick, red] (g) --  (v3);
					\draw[very thick, red] (h) --  (v3);
					\draw[very thick, red] (g) --  (v4);
					\draw[very thick, red] (h) --  (v4);
					
					\draw[thick, dashed, blue] (v1) -- (b);
					\draw[thick, dashed, blue] (v1) -- (c);
					\draw[thick, dashed, blue] (v1) -- (v6);
					\draw[thick, dashed, blue] (b) -- (c);
					\draw[thick, dashed, blue] (b) -- (v6);
					\draw[thick, dashed, blue] (c) -- (v6);
					
					\draw[thick, dashed, blue] (d) -- (v2);
					\draw[thick, dashed, blue] (d) -- (a);
					\draw[thick, dashed, blue] (d) -- (v5);
					\draw[thick, dashed, blue] (v2) -- (a);
					\draw[thick, dashed, blue] (v2) -- (v5);
					\draw[thick, dashed, blue] (a) -- (v5);
					
					\draw[thick, dashed, blue] (v1) -- (a);
					\draw[thick, dashed, blue] (v1) -- (v5);
					\draw[thick, dashed, blue] (b) -- (a);
					\draw[thick, dashed, blue] (b) -- (v5);
					
					\draw[thick, dashed, blue] (c) -- (d);
					\draw[thick, dashed, blue] (c) -- (v2);
					\draw[thick, dashed, blue] (v6) -- (d);
					\draw[thick, dashed, blue] (v6) -- (v2);
					
					\draw (x) -- (v1);
					\draw (x) -- (v6);
					\draw (x) -- (v2);
					\draw (x) -- (v5);
					\draw (v1) -- (g);
					\draw (v1) -- (v4);
					\draw (v1) -- (e);
					\draw (b) -- (h);
					\draw (b) -- (v3);
					\draw (b) -- (f);
					\draw (b) -- (u);
					\draw (g) -- (c);
					\draw (g) -- (v2);
					\draw (g) -- (a);
					\draw (h) -- (v6);
					\draw (h) -- (d);
					\draw (h) -- (v5);
					\draw (c) -- (v3);
					\draw (c) -- (e);
					\draw (c) -- (u);
					\draw (v6) -- (v4);
					\draw (v6) -- (f);
					\draw (v4) -- (v2);
					\draw (v4) -- (v5);
					\draw (v3) -- (d);
					\draw (v3) -- (a);
					\draw (v2) -- (e);
					\draw (d) -- (f);
					\draw (d) -- (u);
					\draw (f) -- (v5);
					\draw (e) -- (a);
					\draw (a) -- (u);
					
				\end{tikzpicture}
			\end{adjustbox}
			\caption{The unique strictly Neumaier graph on 16 vertices. Its maximum cliques are indicated by thick and dashed lines.}
			\label{fig:neu:ng16complete}
		\end{figure}
		
		\begin{case}{1.2}
			$|S| = 3$.
		\end{case}
		\noindent If~$|S| = 3$, there is exactly one vertex in~$\Gamma_2(u)$ that is adjacent to~$a$ and not to~$c$ and exactly one vertex in~$\Gamma_2(u)$ that is adjacent to~$c$ and not to~$a$. Observe that if~$b$ or~$d$ is adjacent to one of these vertices it must also be adjacent to the other, since it must have the same number of common neighbors with~$a$ and~$c$. There are, up to symmetry, five ways to choose~$S$, namely~$\{v_1, v_2, v_3\}, \{v_1, v_2, v_5\}, \{v_3, v_4, v_5\}, \{v_1, v_3, v_5\}$ and~$\{v_1, v_3, v_4\}$, which also uniquely determine the remaining two vertices that are adjacent to~$a$ or~$c$ (note that we could interchange~$a$ and~$c$, since they have the same set of neighbors up to this point). Each of these cases is illustrated in Figure~\ref{fig:neu:scase1sub2a}-\ref{fig:neu:scase1sub2e}. Below, we derive the remaining edges of~$b$ and~$d$ for all of them.
		
		\begin{figure}[!htb]
			\centering
			\begin{subfigure}[b]{0.19\textwidth}
				\centering
				\begin{adjustbox}{width=0.95\textwidth}
					\begin{tikzpicture}[scale=0.4]
						\foreach \x in {0,60,120,180,240,300}{
							\draw[fill=black] (\x:3) circle (4pt);
						}
						\node at (120:4) {$ac$};
						\node at (60:4){$ac$};
						\node at (180:4.5){$abcd$};
						\node at (0:4.5){};
						\node at (240:4){$abd$};
						\node at (300:4){$bcd$};
						
						\draw[dashed] (111:4.5) to (110:4.8);
						\draw[dashed] (69:4.5) to (70:4.8);
						\draw[dashed] (110:4.8) to (70:4.8);
						\node at (90:5.2){$b,d$};
						
						\node[draw=none,fill=none] at(0:5.5){};
						\node[draw=none,fill=none] at(180:5.5){};
					\end{tikzpicture}
				\end{adjustbox}
				\caption{}
				\label{fig:neu:scase1sub2a}
			\end{subfigure}
			\begin{subfigure}[b]{0.19\textwidth}
				\centering
				\begin{adjustbox}{width=0.95\textwidth}
					\begin{tikzpicture}[scale=0.4]
						\foreach \x in {0,60,120,180,240,300}{
							\draw[fill=black] (\x:3) circle (4pt);
						}
						\node at (120:4) {$ac$};
						\node at (60:4){$ac$};
						\node at (180:4.4){$abd$};
						\node at (0:4.4){$bcd$};
						\node at (240:4){$abcd$};
						\node at (300:4){};
						
						\draw[dashed] (111:4.5) to (110:4.8);
						\draw[dashed] (69:4.5) to (70:4.8);
						\draw[dashed] (110:4.8) to (70:4.8);
						\node at (90:5.2){$b,d$};
						
						\node[draw=none,fill=none] at(0:5.5){};
						\node[draw=none,fill=none] at(180:5.5){};
					\end{tikzpicture}
				\end{adjustbox}
				\caption{}
				\label{fig:neu:scase1sub2b}
			\end{subfigure}
			\begin{subfigure}[b]{0.19\textwidth}
				\centering
				\begin{adjustbox}{width=0.95\textwidth}
					\begin{tikzpicture}[scale=0.4]
						\foreach \x in {0,60,120,180,240,300}{
							\draw[fill=black] (\x:3) circle (4pt);
						}
						\node at (120:4) {$a$};
						\node at (60:4){$c$};
						\node at (180:4.2){$ac$};
						\node at (0:4.2){$ac$};
						\node at (240:4){$ac$};
						\node at (300:4){};
						\node[draw=none,fill=none] at(0:5.5){};
						\node[draw=none,fill=none] at(180:5.5){};
					\end{tikzpicture}
				\end{adjustbox}
				\caption{}
				\label{fig:neu:scase1sub2c}
			\end{subfigure}\\
            \vspace{0.5cm}
			\begin{subfigure}[b]{0.38\textwidth}
				\centering
				\begin{adjustbox}{width=0.475\textwidth}
					\begin{tikzpicture}[scale=0.4]
						\foreach \x in {0,60,120,180,240,300}{
							\draw[fill=black] (\x:3) circle (4pt);
						}
						\node at (120:4) {$acd$};
						\node at (60:4){$ab$};
						\node at (180:4.5){$abcd$};
						\node at (0:4.2){$bc$};
						\node at (240:4){$abcd$};
						\node at (300:4){$d$};
						\node[draw=none,fill=none] at(0:5.5){};
						\node[draw=none,fill=none] at(180:5.5){};
					\end{tikzpicture}
				\end{adjustbox}
				\begin{adjustbox}{width=0.475\textwidth}
					\begin{tikzpicture}[scale=0.4]
						\foreach \x in {0,60,120,180,240,300}{
							\draw[fill=black] (\x:3) circle (4pt);
						}
						\node at (120:4) {$abcd$};
						\node at (60:4){$a$};
						\node at (180:4.5){$abcd$};
						\node at (0:4){$c$};
						\node at (240:4){$abcd$};
						\node at (300:4){$bd$};
						\node[draw=none,fill=none] at(0:5.5){};
						\node[draw=none,fill=none] at(180:5.5){};
					\end{tikzpicture}
				\end{adjustbox}
				\caption{}
				\label{fig:neu:scase1sub2d}
			\end{subfigure}
			\begin{subfigure}[b]{0.19\textwidth}
				\centering
				\begin{adjustbox}{width=0.95\textwidth}
					\begin{tikzpicture}[scale=0.4]
						\foreach \x in {0,60,120,180,240,300}{
							\draw[fill=black] (\x:3) circle (4pt);
						}
						\node at (120:4) {$abcd$};
						\node at (60:4){};
						\node at (180:4.2){$ac$};
						\node at (0:4.2){$ac$};
						\node at (240:4){$abd$};
						\node at (300:4){$bcd$};
						\draw[dashed] (12:3) to (0:3);
						\draw[dashed] (180:3) to (168:3);
						\draw[dashed] (168:3) to (12:3);
						\node at (90:1.3){$b,d$};
						\node[draw=none,fill=none] at(0:5.5){};
						\node[draw=none,fill=none] at(180:5.5){};
					\end{tikzpicture}
				\end{adjustbox}
				\caption{}
				\label{fig:neu:scase1sub2e}
			\end{subfigure}
			\caption{The edges between~$\Gamma_2(u)$ and~$\{a, b, c, d\}$ when~$|S| = 3$ .}
			\label{fig:neu:scases1sub2}
		\end{figure}
  
		\begin{dense_enumerate}{label=(\alph*)}
			\item Suppose that~$S = \{v_1, v_2, v_3\}$. Vertices~$b$ and~$d$ are adjacent to only one vertex in~$\{v_1, v_2\}$, hence their remaining neighbors in~$\Gamma_2(u)$ must be adjacent to~$a$ and~$c$. If they both have the same neighbor in~$\{v_1, v_2\}$, this leads to degree sequence~$(1,2,3,3,4,5)$ for the subgraph induced by~$\Gamma_2(u)$, and otherwise we get~$(1,3,3,3,3,5)$. The first sequence does not create a valid Neumaier graph, as Lemma~\ref{lem:neu:nonexistence}$\ref{lem:neu:nonexistence3}$ implies that~$\Gamma_2(u)$ contains a 4-clique. This clique is not 2-regular, because~$u$ has no neighbors in it. In the second case, we may assume that~$b\sim v_1$ and~$d\sim v_2$, since~$b$ and~$d$ have the same neighborhood up to this point. We know from Lemma~\ref{lem:neu:nonexistence}$\ref{lem:neu:nonexistence1}$ that there exists a unique graph with degree sequence $(1,3,3,3,3,5)$. Moreover, in this graph, the vertices of degree three are symmetric under its automorphism group. Therefore, we can uniquely embed this graph into $\Gamma_2(u)$ up to isomorphism. It is easily checked that the resulting graph is strictly Neumaier with eight 2-regular 4-cliques. The full graph is shown in Figure~\ref{fig:neu:ng16complete}.
			\item Let~$S = \{v_1, v_2, v_5\}$. Analogous to the previous case,~$b$ and~$d$ must share all neighbors of~$a$ and~$c$ among~$v_3, \dots, v_6$. As a result,~$v_6$ is not adjacent to~$a, b, c$ and~$d$, so it must have six neighbors in~$\Gamma_2(u)$, but this is not possible.
			\item If~$S=\{v_3, v_4, v_5\}$, then~$a$ and~$c$ have distinct neighbors in~$\{v_1, v_2\}$. Vertices~$b$ and~$d$ also have one neighbor in~$\{v_1, v_2\}$, but by the above observation they must be adjacent to both~$v_1$ and~$v_2$ or neither, a contradiction.
			\item Assume that~$S = \{v_1, v_3, v_5\}$. Then~$a\sim v_2$ and~$c\sim v_4$ (or vice versa), so~$b$ (and also~$d$) is either adjacent to both~$v_2$ and~$v_4$ or none of these vertices. First, assume that~$b\sim v_2, v_4$. In that case, we also know that~$b\sim v_3, v_5$, since it has three common neighbors with~$a$ and~$c$. Vertex~$v_6$ must then be adjacent to~$d$, since it has at most 5 neighbors in~$\Gamma_2(u)$, and it has degree nine. This implies that~$d\sim v_1, v_3, v_5$. Therefore,~$\Gamma_2(u)$ must have degree sequence~$(1,2,3,3,4,5)$, so by Lemma~\ref{lem:neu:nonexistence}$\ref{lem:neu:nonexistence3}$ it contains a 4-clique. This violates 2-regularity, because~$u$ has no neighbors in this clique. Note that the case~$d\sim v_2, v_4$ is entirely symmetric, so we may now assume that both~$b$ and~$d$ are not adjacent to~$v_2$ and~$v_4$. However, then we must have~$b,d\sim v_1, v_3, v_5, v_6$, which implies that the graph spanned by~$\Gamma_2(u)$ has degree sequence~$(1,2,2,4,4,5)$, contradicting Lemma~\ref{lem:neu:nonexistence}$\ref{lem:neu:nonexistence2}$.
			\item Finally, let~$S = \{v_1, v_3, v_4\}$. Then~$v_5$ is adjacent to~$a$ but not to~$c$, and~$v_6$ is adjacent to~$c$ but not to~$a$ (or vice versa, but this case is symmetric). By the pigeonhole principle,~$b$ and~$d$ have at least one neighbor in~$\{v_5, v_6\}$, so they must be adjacent to both~$v_5$ and~$v_6$. This means they~$b$ and~$d$ will have two common neighbors with~$a$ as well as with~$c$ within~$\{v_3, v_4, v_5, v_6\}$, so they must also both be adjacent to~$v_1$. Then~$v_2$ is not adjacent to~$a$,~$b$,~$c$ and~$d$, which means that it has six neighbors in~$\Gamma_2(u)$, a contradiction.
		\end{dense_enumerate}
		
		\begin{case}{1.3}
			$|S| = 4$.
		\end{case}
		\noindent In this case,~$\Gamma(a)\cap \Gamma_2(u) = S = \Gamma(c)\cap \Gamma_2(u)$, so three vertices of~$S$ are in~$\{v_1, v_2, v_3, v_4\}$ and the other in~$\{v_5, v_6\}$. Without loss of generality, we assume that either~$S = \{v_1, v_3, v_4, v_5\}$ or~$S=\{v_1, v_2, v_3, v_5\}$. These two cases are shown in Figures~\ref{fig:neu:scase1sub3a} and~\ref{fig:neu:scase1sub3b}, respectively.
		
		\begin{figure}[!htb]
			\centering
			\begin{subfigure}[b]{0.19\textwidth}
				\centering
				\begin{adjustbox}{width=0.95\textwidth}
					\begin{tikzpicture}[scale=0.4]
						\foreach \x in {0,60,120,180,240,300}{
							\draw[fill=black] (\x:3) circle (4pt);
						}
						\node at (120:4) {$acd$};
						\node at (60:4){$b$};
						\node at (180:4.5){$abcd$};
						\node at (0:4.4){$abc$};
						\node at (240:4){$abc$};
						\node at (300:4){$d$};
						
						\draw[dashed] (8:2.6) to (0:3);
						\draw[dashed] (232:2.6) to (240:3);
						\draw[dashed] (8:2.6) to (232:2.6);
						\node at (-45:0.2){$d$};
						
						\node[draw=none,fill=none] at(0:5.5){};
						\node[draw=none,fill=none] at(180:5.5){};
					\end{tikzpicture}
				\end{adjustbox}
				\caption{}
				\label{fig:neu:scase1sub3a}
			\end{subfigure}
			\begin{subfigure}[b]{0.19\textwidth}
				\centering
				\begin{adjustbox}{width=0.95\textwidth}
					\begin{tikzpicture}[scale=0.4]
						\foreach \x in {0,60,120,180,240,300}{
							\draw[fill=black] (\x:3) circle (4pt);
						}
						\node at (120:4) {$ac$};
						\node at (60:4){$ac$};
						\node at (180:4.5){$abcd$};
						\node at (0:4.5){};
						\node at (240:4){$abcd$};
						\node at (300:4){};
						
						\draw[dashed] (0:3) to (0:2.4);
						\draw[dashed] (300:2.4) to (300:3);
						\draw[dashed] (0:2.4) to (300:2.4);
						\node at (330:1){$b,d$};
						
						\draw[dashed] (111:4.5) to (110:4.8);
						\draw[dashed] (69:4.5) to (70:4.8);
						\draw[dashed] (110:4.8) to (70:4.8);
						\node at (90:5.2){$b,d$};
						
						\node[draw=none,fill=none] at(0:5.5){};
						\node[draw=none,fill=none] at(180:5.5){};
					\end{tikzpicture}
				\end{adjustbox}
				\caption{}
				\label{fig:neu:scase1sub3b}
			\end{subfigure}
			\caption{The edges between~$\Gamma_2(u)$ and~$\{a, b, c, d\}$ when~$|S| = 4$.}
			\label{fig:neu:scases1sub3}
		\end{figure}
		
		\begin{dense_enumerate}{label=(\alph*)}
			\item First consider~$S = \{v_1, v_3, v_4, v_5\}$. Vertex~$v_2$ cannot have degree six in the subgraph induced on~$\Gamma_2(u)$, so it must have another neighbor, without loss of generality~$b$. Since~$b$ has three common neighbors with~$a$ and~$c$, this implies~$b\sim v_3, v_4, v_5$. By the same argument,~$d\sim v_6$. Then~$d\sim v_1$, and it has two neighbors among~$\{v_3, v_4, v_5\}$, because it must have three common neighbors with~$a$ and~$c$. In particular, it must be adjacent to~$v_3$ or~$v_4$. Without loss of generality, let~$d\sim v_3$. However, then the subgraph induced by~$\Gamma_2(u)$ has two vertices of degree five,~$v_2$ and~$v_6$, which must both be adjacent to~$v_3$, which has degree one. This is clearly not possible. 
			\item If~$S=\{v_1, v_2, v_3, v_5\}$, then~$b$ and~$d$ can only have three common neighbors with~$a$ and~$c$ if~$b,d\sim v_3,v_5$ and both have one neighbor in~$\{v_1, v_2\}$. Also,~$v_6$ is adjacent to at least one of~$b$,~$d$, otherwise it must have degree six in~$\Gamma_2(u)$. If~$b,d\sim v_6$, then~$v_4$ and~$v_6$ have degree five and four in the subgraph induced on~$\Gamma_2(u)$, whereas the reverse is true when only one of them is adjacent to~$v_6$ and the other to~$v_4$. If~$b$ and~$d$ have the same neighbor in~$\{v_1, v_2\}$, this leads to degree sequence~$(1,2,2,4,4,5)$ and otherwise we get~$(1,2,3,3,4,5)$. The first degree sequence cannot exist due to Lemma~\ref{lem:neu:nonexistence}$\ref{lem:neu:nonexistence2}$, the second is ruled out by Lemma~\ref{lem:neu:nonexistence}$\ref{lem:neu:nonexistence3}$.
		\end{dense_enumerate}
		
		\begin{case}{2} The edges between~$B$ and~$\Gamma_2(u)$ are as in Figure~\ref{fig:neu:bowtiecomplete2}.
		\end{case}
		\noindent Recall that in this case,~$v_3, v_4\sim e, f, g, h$, while~$v_1$,~$v_2$,~$v_5$ and~$v_6$ are adjacent to~$x$, as well as to~$e$ and~$g$,~$g$ and~$f$,~$f$ and~$h$ and~$h$ and~$e$, respectively. This implies the following.
		
		Both~$a$ and~$c$ are adjacent to~$e$ as well as to~$g$. To satisfy~$4$-edge-regularity, they both must have three common neighbors with both~$e$ and~$g$ in~$\Gamma_2(u)$, as they only have~$u$ as a common neighbor in~$\{u\}\cup\Gamma(u)$. Therefore,~$a$ and~$c$ must have exactly three neighbors in the sets~$\{v_1, v_2, v_3, v_4\}$ and~$\{v_1, v_3, v_4, v_6\}$. Similarly,~$b$ and~$d$ have three neighbors in~$\{v_3, v_4, v_5, v_6\}$ and in~$\{v_2, v_3, v_4, v_5\}$. Analogous to Case 1, both~$a$ and~$c$ must also have three common neighbors in~$\Gamma_2(u)$ with~$b$ as well as with~$d$.
		
		Note that the above implies that either~$\{v_1, v_3, v_4\}\subset\Gamma(a)$ or~$|\{v_1, v_3, v_4\}\cap\Gamma(a)|=2$. The same is true for~$c$. Therefore, it suffices to consider the following three cases:~$\{v_1, v_3, v_4\}\subset\Gamma(a)\cap \Gamma(c)$,~$\{v_1, v_3, v_4\}\subset\Gamma(a)$ and~$|\{v_1, v_3, v_4\}\cap\Gamma(c)|=2$ (note that we could interchange~$a$ and~$c$ here, since these cases are symmetric) and~$|\{v_1, v_3, v_4\}\cap\Gamma(a)|=|\{v_1, v_3, v_4\}\cap\Gamma(c)|=2$.
		
		\begin{case}{2.1}$\{v_1, v_3, v_4\}\subset\Gamma(a)\cap \Gamma(c)$.
		\end{case}
		\noindent If~$\{v_1, v_3, v_4\}\subset\Gamma(a)\cap \Gamma(c)$,~$a$ and~$c$ cannot be adjacent to~$v_2$ or~$v_6$, otherwise they would have more than three common neighbors with~$g$ or~$e$ in~$\Gamma_2(u)$. Therefore, their final neighbor is~$v_5$. Moreover, both~$b$ and~$d$ cannot be adjacent to~$v_6$. If they are, they must also be adjacent to~$v_2$ to ensure 4-edge-regularity with~$f$, hence they cannot have three common neighbors with~$a$ and~$c$ in~$\Gamma_2(u)$. Then~$v_6$ has degree six in~$\Gamma_2(u)$, which is not possible.
		
		\begin{case}{2.2}$\{v_1, v_3, v_4\}\subset\Gamma(a)$ and~$|\{v_1, v_3, v_4\}\cap\Gamma(c)|=2$.
		\end{case}
		\noindent Vertices~$v_3$ and~$v_4$ have the same set of neighbors in~$B$. Therefore, we may assume without loss of generality that~$|\{v_1, v_3, v_4\}\cap\Gamma(c)|$ is either~$\{v_1, v_3\}$ or~$\{v_3, v_4\}$. Both options are shown in Figure~\ref{fig:neu:scases2sub2}. Note that~$a\sim v_5$, since it already has three neighbors in~$\Gamma_2(u)$ with~$e$ and~$g$. Analogously, also note that ~$c\sim v_2, v_6$.
		
		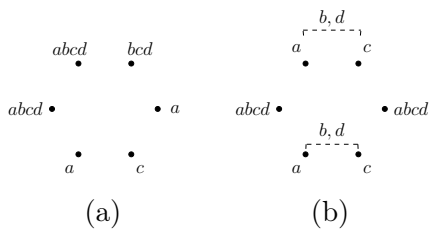
\begin{figure}[!htb]
			\centering
			\begin{subfigure}[b]{0.19\textwidth}
				\centering
				\begin{adjustbox}{width=0.95\textwidth}
					\begin{tikzpicture}[scale=0.4]
						\foreach \x in {0,60,120,180,240,300}{
							\draw[fill=black] (\x:3) circle (4pt);
						}
						\node at (120:4) {$abcd$};
						\node at (60:4){$bcd$};
						\node at (180:4.5){$abcd$};
						\node at (0:4){$a$};
						\node at (240:4){$a$};
						\node at (300:4){$c$};
						
						\node[draw=none,fill=none] at(0:5.5){};
						\node[draw=none,fill=none] at(180:5.5){};
					\end{tikzpicture}
				\end{adjustbox}
				\caption{}
				\label{fig:neu:scase2sub2a}
			\end{subfigure}
			\begin{subfigure}[b]{0.19\textwidth}
				\centering
				\begin{adjustbox}{width=0.95\textwidth}
					\begin{tikzpicture}[scale=0.4]
						\foreach \x in {0,60,120,180,240,300}{
							\draw[fill=black] (\x:3) circle (4pt);
						}
						\node at (120:4) {$a$};
						\node at (60:4){$c$};
						\node at (180:4.5){$abcd$};
						\node at (0:4.5){$abcd$};
						\node at (240:4){$a$};
						\node at (300:4){$c$};
						
						\draw[dashed] (111:4.5) to (110:4.8);
						\draw[dashed] (69:4.5) to (70:4.8);
						\draw[dashed] (110:4.8) to (70:4.8);
						\node at (90:5.2){$b,d$};
						
						\draw[dashed] (234:2.5) to (240:3);
						\draw[dashed] (306:2.5) to (300:3);
						\draw[dashed] (306:2.5) to (234:2.5);
						\node at (270:1.3){$b,d$};
						
						\node[draw=none,fill=none] at(0:5.5){};
						\node[draw=none,fill=none] at(180:5.5){};
					\end{tikzpicture}
				\end{adjustbox}
				\caption{}
				\label{fig:neu:scase2sub2b}
			\end{subfigure}
			\caption{The edges between~$\Gamma_2(u)$ and~$\{a, b, c, d\}$ when~$\{v_1, v_3, v_4\}\subset\Gamma(a)$ and $|\{v_1, v_3, v_4\}\cap\Gamma(c)|=2$.}
			\label{fig:neu:scases2sub2}
		\end{figure}
		
		\begin{dense_enumerate}{label=(\alph*)}
			\item Suppose that~$|\{v_1, v_3, v_4\}\cap\Gamma(c)| =  \{v_1, v_3\}$. Since~$b$ and~$d$ have three common neighbors with~$a$ and~$c$ in~$\G_2(u)$, they must be adjacent to both~$v_1$ and~$v_3$. They are also adjacent to~$v_2$, because they have three neighbors in~$\{v_2, v_3, v_4, v_5\}$. However, then they cannot have three neighbors in~$\{v_3, v_4, v_5,v_6\}$, a contradiction.
			\item If~$|\{v_1, v_3, v_4\}\cap\Gamma(c)| = \{v_3, v_4\}$, then we also know that~$c\sim v_2, v_6$. Analogous to Case 2.2(a), we have~$b,d\sim v_3, v_4$, so~$v_3$ and~$v_4$ have degree one in the subgraph induced by~$\Gamma_2(u)$. However, the other four vertices of this subgraph can span at most six edges. Hence~$\Gamma_2(u)$ has at most eight edges, while it should have nine, a contradiction.
		\end{dense_enumerate}
		
		\begin{case}{2.3}$|\{v_1, v_3, v_4\}\cap\Gamma(a)|=|\{v_1, v_3, v_4\}\cap\Gamma(c)|=2$.
		\end{case}
		\noindent In this case,~$e$ and~$g$ must have two common neighbors with~$a$ and~$c$ outside of~$\{v_1, v_3, v_4\}$, so we know that~$a,c\sim v_2, v_6$. Since neither~$a$ nor~$c$ is adjacent to~$v_{5}$, and~$v_{5}$ cannot have degree six in~$\Gamma_2(u)$, we may assume without loss of generality that~$b\sim v_5$. Within~$\{v_1, v_3, v_4\}$,~$a$ and~$c$ have one or two common neighbors. If they have one common neighbor in~$\{v_1, v_3, v_4\}$, then this is, up to symmetry, either~$v_1$ or~$v_3$. If~$a$ and~$c$ have two common neighbors in~$\{v_1, v_3, v_4\}$, then these are, up to symmetry, either~$v_1$ and~$v_3$ or~$v_3$ and~$v_4$. Figure~\ref{fig:neu:scases2sub3} illustrates all options.
		
		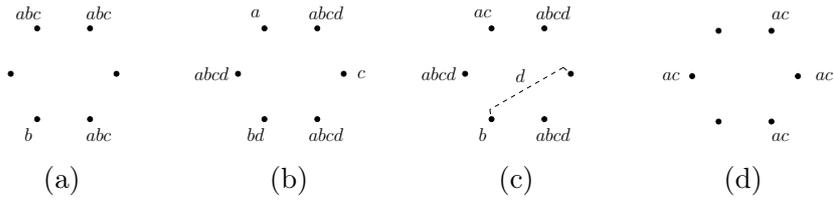
\begin{figure}[!htb]
			\centering
			\begin{subfigure}[b]{0.19\textwidth}
				\centering
				\begin{adjustbox}{width=0.95\textwidth}
					\begin{tikzpicture}[scale=0.4]
						\foreach \x in {0,60,120,180,240,300}{
							\draw[fill=black] (\x:3) circle (4pt);
						}
						\node at (120:4) {$abc$};
						\node at (60:4){$abc$};
						\node at (180:4.5){};
						\node at (0:4.5){};
						\node at (240:4){$b$};
						\node at (300:4){$abc$};
						
						\node[draw=none,fill=none] at(0:5.5){};
						\node[draw=none,fill=none] at(180:5.5){};
					\end{tikzpicture}
				\end{adjustbox}
				\caption{}
				\label{fig:neu:scase2sub3a}
			\end{subfigure}
			\begin{subfigure}[b]{0.19\textwidth}
				\centering
				\begin{adjustbox}{width=0.95\textwidth}
					\begin{tikzpicture}[scale=0.4]
						\foreach \x in {0,60,120,180,240,300}{
							\draw[fill=black] (\x:3) circle (4pt);
						}
						\node at (120:4) {$a$};
						\node at (60:4){$abcd$};
						\node at (180:4.5){$abcd$};
						\node at (0:4){$c$};
						\node at (240:4){$bd$};
						\node at (300:4){$abcd$};
						
						\node[draw=none,fill=none] at(0:5.5){};
						\node[draw=none,fill=none] at(180:5.5){};
					\end{tikzpicture}
				\end{adjustbox}
				\caption{}
				\label{fig:neu:scase2sub3b}
			\end{subfigure}
			\begin{subfigure}[b]{0.19\textwidth}
				\centering
				\begin{adjustbox}{width=0.95\textwidth}
					\begin{tikzpicture}[scale=0.4]
						\foreach \x in {0,60,120,180,240,300}{
							\draw[fill=black] (\x:3) circle (4pt);
						}
						\node at (120:4) {$ac$};
						\node at (60:4){$abcd$};
						\node at (180:4.5){$abcd$};
						\node at (0:4.4){};
						\node at (240:4){$b$};
						\node at (300:4){$abcd$};
						
						\draw[dashed] (8:2.6) to (0:3);
						\draw[dashed] (232:2.6) to (240:3);
						\draw[dashed] (8:2.6) to (232:2.6);
						\node at (-45:0.2){$d$};
						
						\node[draw=none,fill=none] at(0:5.5){};
						\node[draw=none,fill=none] at(180:5.5){};
					\end{tikzpicture}
				\end{adjustbox}
				\caption{}
				\label{fig:neu:scase2sub3c}
			\end{subfigure}
			\begin{subfigure}[b]{0.19\textwidth}
				\centering
				\begin{adjustbox}{width=0.95\textwidth}
					\begin{tikzpicture}[scale=0.4]
						\foreach \x in {0,60,120,180,240,300}{
							\draw[fill=black] (\x:3) circle (4pt);
						}
						\node at (120:4) {};
						\node at (60:4){$ac$};
						\node at (180:4.2){$ac$};
						\node at (0:4.5){$ac$};
						\node at (240:4){};
						\node at (300:4){$ac$};
						
						\node[draw=none,fill=none] at(0:5.5){};
						\node[draw=none,fill=none] at(180:5.5){};
					\end{tikzpicture}
				\end{adjustbox}
				\caption{}
				\label{fig:neu:scase2sub3d}
			\end{subfigure}
			\caption{The edges between~$\Gamma_2(u)$ and~$\{a, b, c, d\}$ when~$|\{v_1, v_3, v_4\}\cap\Gamma(a)|=|\{v_1, v_3, v_4\}\cap\Gamma(c)|=2$.}
			\label{fig:neu:scases2sub3}
		\end{figure}
		
		\begin{dense_enumerate}{label=(\alph*)}
			\item If~$a,c\sim v_1$ and~$v_1$ is the only common neighbor of~$a$ and~$c$ in~$\{v_1, v_3, v_4\}$, then~$b$ must be adjacent to~$v_1$,~$v_2$ and~$v_6$ in order to have three common neighbors with~$a$ and~$c$ in~$\Gamma_2(u)$. This is a contradiction, since we know it has three neighbors in~$\{v_3, v_4, v_5, v_6\}$.
			\item If~$a,c\sim v_3$ and~$v_3$ is the only common neighbor of~$a$ and~$c$ in~$\{v_1, v_3, v_4\}$, then we may assume without loss of generality that~$a\sim v_1$ and~$c\sim v_4$. Then~$b\sim v_2, v_3, v_6$, as~$b$ has three common neighbors with~$a$,~$c$,~$f$ and~$h$ in~$\Gamma_2(u)$. If~$d\sim v_5$, it has the same neighbors as~$b$ and the subgraph induced by~$\Gamma_2(u)$ has degree sequence~$(1,2,2,4,4,5)$. This contradicts Lemma~\ref{lem:neu:nonexistence}$\ref{lem:neu:nonexistence2}$. Otherwise, the only option which ensures 4-edge-regularity with~$f$ is~$d\sim v_2, v_3, v_4, v_6$. However, in this case~$c$ and~$d$ have five common neighbors, a contradiction.
			\item If~$a,c\sim v_1,v_2,v_3,v_6$, then~$b\sim v_2, v_3, v_6$, as~$b$ has three common neighbors with~$a$,~$c$,~$f$ and~$h$ in~$\Gamma_2(u)$. If~$d\sim v_5$, it has the same neighbors as~$b$. Otherwise, the only option is~$d\sim v_2, v_3, v_4, v_6$. In both cases,~$\Gamma_2(u)$ has degree sequence~$(1,2,2,4,4,5)$, contradicting Lemma~\ref{lem:neu:nonexistence}$\ref{lem:neu:nonexistence2}$.
			\item Let~$a, c\sim v_2, v_3, v_4, v_6$. If~$b\sim v_1$, it must also be adjacent to~$v_3$ and~$v_4$ to ensure 4-edge-regularity with~$f$ and~$h$. This cannot be the case, because then~$b$ does not have three common neighbors with~$a$ and~$c$ in~$\Gamma_2(u)$. The same holds for~$d$. However, if~$b,d\not\sim v_1$, then~$v_1$ must have degree six in~$\Gamma_2(u)$, a contradiction. \qedhere
		\end{dense_enumerate}
	\end{proof}
	
	\section{A new strictly Neumaier graph on 25 vertices}
 \label{sec:neu:new25}
	
	In this section, we construct a strictly Neumaier graph with parameters~$(25,12,5;\allowbreak 2,5)$ from a Latin square of order five. At the time of writing, this appears to be the first strictly Neumaier graph with odd order and~$e>1$, and also the first one\footnote{No other examples have been published to our knowledge, but when finishing this paper we were informed by Sergey Goryainov that he constructed 18 strictly Neumaier graphs with parameters $(64,28,12;3,8)$ and one strictly Neumaier graph with parameters $(64,42,26;5,8)$, alongside some strictly Neumaier graphs with parameters $(64,35,18;4,8)$ \cite{goryainovcommunication}.} with $1\neq e\neq \frac{s}{2}$. It is also the first (and so far only, to our knowledge) strictly Neumaier graph with $v$ odd, that is not vertex-transitive; several non-vertex-transitive strictly Neumaier graphs with $v$ even were constructed in \cite[Section 4.4.2]{evansphd}. Moreover, the parameter set~$(25,12,5;2,5)$ is the only one for which at this moment only a non-vertex-transitive example is known: for all other parameter sets for which existence is established, a vertex-transitive example is known. Our construction requires the following preliminary definitions.
	
	A \textit{Latin square} of order~$n$ is an~$n\times n$ grid where every cell has a label from the set~$\{1,\dots,n\}$, often referred to as its color, such that every value appears exactly once in every row and column. With a Latin square, one can associate a graph on~$n^2$ vertices, corresponding to the cells, which are adjacent if they share a row, column or color. These strongly regular graphs are referred to as \textit{Latin square graphs} and are Neumaier with parameters~$(n^2,3(n-1),n;2,n)$ for any given Latin square. In the remainder of this section, we label the cell of a Latin square in row~$i$ and column~$j$, as well as the corresponding vertex in the Latin square graph, as~$x_{ij}$.
    \par Let~$L$ be the Latin square of order 5 shown in Figure~\ref{fig:neu:LS25}. Note that~$L$ has a cell~$x_{11}$ of color 1, which is contained in four $2\times 2$ Latin subsquares, one for each other color class.

\begin{figure}[!htb]
 \centering
	   \begin{adjustbox}{width=0.3\textwidth}
        \begin{tikzpicture}
        \matrix[matrix of nodes,nodes={inner sep=0pt,text width=.5cm,align=center,minimum height=.4915cm}]{
            |[fill=black!40]|1&|[fill=red!40]|2&|[fill=darkorange!40]|3&|[fill=LimeGreen!40]|4&|[fill=cyan!40]|5 \\
            |[fill=red!40]|2&|[fill=red!40]|1&4&5&3 \\
            |[fill=darkorange!40]|3&5&|[fill=darkorange!40]|1&2&4 \\
            |[fill=LimeGreen!40]|4&3&5&|[fill=LimeGreen!40]|1&2 \\
            |[fill=cyan!40]|5&4&2&3&|[fill=cyan!40]|1 \\};
            \draw[step=0.5cm,color=black,xshift=-0.25cm, yshift=-0.25cm] (-1,-1) grid (1.5,1.5);
        \end{tikzpicture}
      \end{adjustbox}
	   \caption{A Latin square of order five. Cell~$x_{11}$ is contained in four~$2\times 2$ Latin subsquares, each indicated with a different color.}
		\label{fig:neu:LS25}
	\end{figure}
 
    We define the following $3\times3$ matrices:

\begin{align*}
    I&=\begin{pmatrix}
        1&0&0\\
        0&1&0\\
        0&0&1
    \end{pmatrix}&
    A&=\begin{pmatrix}
        0&1&0\\
        0&0&1\\
        1&0&0
    \end{pmatrix}&
    B&=\begin{pmatrix}
        0&1&1\\
        1&0&1\\
        1&1&0
    \end{pmatrix}=A+A^{t}\\
    E_1&=\begin{pmatrix}
        1&1&1\\
        0&0&0\\
        0&0&0
    \end{pmatrix}&
    E_2&=\begin{pmatrix}
        0&0&0\\
        1&1&1\\
        0&0&0
    \end{pmatrix}&
    E_3&=\begin{pmatrix}
        0&0&0\\
        0&0&0\\
        1&1&1
    \end{pmatrix}\\
    \overline{E_1}&=\begin{pmatrix}
        0&0&0\\
        1&1&1\\
        1&1&1
    \end{pmatrix}&
    \overline{E_2}&=\begin{pmatrix}
        1&1&1\\
        0&0&0\\
        1&1&1
    \end{pmatrix}&
    \overline{E_3}&=\begin{pmatrix}
        1&1&1\\
        1&1&1\\
        0&0&0
    \end{pmatrix}
\end{align*}
and let $O$ be the $3\times3$ zero matrix. Let $\bm{j}$ and $\bm{0}$ be the all-one and all-zero row vector of length 3, respectively. We can represent the Latin square graph~$\G^{(L)}$ corresponding to~$L$ by the symmetric block matrix~$M$ defined as

\begin{align*}
    M= \begin{pmatrix}
        0 & \bm{j} & \bm{j} & \bm{j} & \bm{j} & \bm{0} & \bm{0} & \bm{0} & \bm{0}\\
         \bm{j}^t& B & E_1 & E_2 & E_3 & \overline{E_1} & \overline{E_2} & \overline{E_3} & O\\
         \bm{j}^t& E_1^t & B & I & I & I+A& A & I & B\\
         \bm{j}^t& E_2^t & I & B & I & I & I+A & A & I+A^{t}\\
         \bm{j}^t& E_3^t & I & I & B & A & I & I+A & I+A\\
         \bm{0}^t& \overline{E_1}^t & I+A^t & I & A^t & B & A & A^{t} & I+A^{t}\\
         \bm{0}^t& \overline{E_2}^t & A^t & I+A^t & I & A^t & B & A & I+A\\
         \bm{0}^t& \overline{E_3}^t & I & A^t & I+A^t & A & A^t & B & B\\
         \bm{0}^t& O & B & I+A & I+A^t& I+A& I+A^t & B & O
    \end{pmatrix}.
\end{align*}
 The first row of~$M$ corresponds to~$x_{11}$. This representation is obtained as follows. 

 Let~$L'$ correspond to one of the $2\times 2$ Latin subsquares of~$L$ containing~$x_{11}$. We will choose~$L'=\{x_{11},x_{12},x_{21},x_{22}\}$, but note that all four choices are equivalent under the automorphism group of~$L$. Let~$C_1,C_2,C_3 \subset V(\G^{(L)})$ be the three~$5$-cliques containing~$x_{11}$. The second block row/column of~$M$ then corresponds to~$L'\setminus \{x_{11}\}$ and the third, fourth and fifth block row/column of~$M$ correspond to~$C_1\setminus L'$,~$C_2\setminus L'$ and~$C_3\setminus L'$. Let~$C'_1,C'_2,C'_3 \subset V(\G^{(L)})$ be the three~$5$-cliques containing two vertices of $L'$ but not~$x_{11}$, and such that $C_i$ and $C'_i$ have no vertices in common. The next three block rows/columns correspond to~$C'_1\setminus L'$,~$C'_2\setminus L'$ and~$C'_3\setminus L'$. The last block row/column of~$M$ then corresponds to the remaining three vertices.
 
 Starting from the graph~$\G^{(L)}$  we can construct a strictly Neumaier graph~$\G^{(25)}$. Note that the subgraph induced on the nine vertices of $(C'_1\cup C'_2\cup C'_3)\setminus L'$ is a $3\times3$ rook's graph, as given in Figure \ref{fig:neu:switching}\subref{fig:neu:preswitch}. The graph~$\G^{(25)}$ is obtained from~$\G^{(L)}$ by removing the edges from three disjoint triangles, namely~$\{x_{24},x_{32},x_{34}\}$, $\{x_{25},x_{42},x_{45}\}$ and $\{x_{23},x_{52},x_{53}\}$, and adding nine edges, thereby creating the disjoint triangles~$\{x_{24},\allowbreak \underline{} x_{42},x_{45}\}$, $\{x_{25},x_{32},x_{53}\}$ and $\{x_{23},x_{42},x_{34}\}$. So, the induced subgraph on $(C'_1\cup C'_2\cup C'_3)\setminus L'$ from Figure \ref{fig:neu:switching}\subref{fig:neu:preswitch} is replaced by the one from Figure \ref{fig:neu:switching}\subref{fig:neu:afterswitch}.

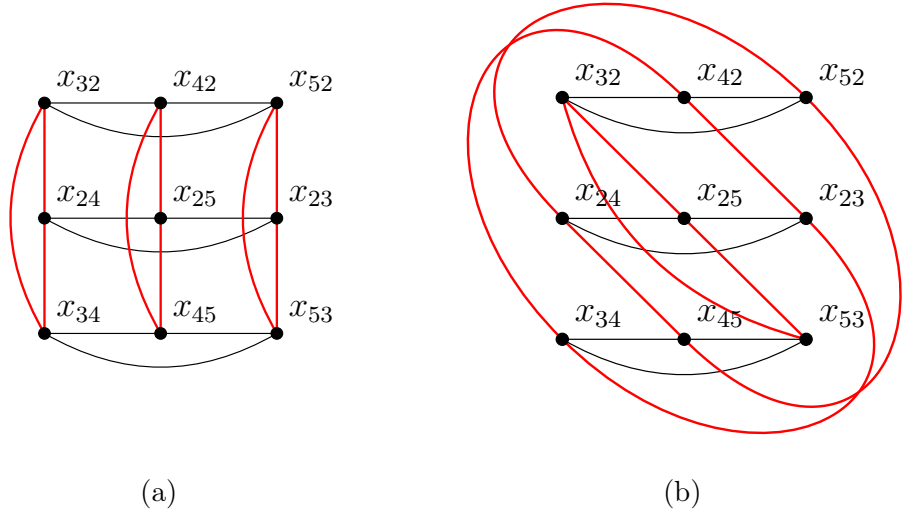
\begin{figure}[htb!]
    \centering
    \begin{subfigure}[b]{0.45\textwidth}
    	\centering
    	\begin{adjustbox}{width=0.95\textwidth}
            \begin{tikzpicture}[scale=0.7]
                \draw (0,0) -- (4,0);
                \draw (0,2) -- (4,2);
                \draw (0,4) -- (4,4);
                \draw (0,0) to [out=330,in=210] (4,0);
                \draw (0,2) to [out=330,in=210] (4,2);
                \draw (0,4) to [out=330,in=210] (4,4);
    
            	\draw[thick,red] (0,0) -- (0,4);
            	\draw[thick,red] (2,0) -- (2,4);
            	\draw[thick,red] (4,0) -- (4,4);
            	\draw[thick,red] (0,0) to [out=120,in=240] (0,4);
            	\draw[thick,red] (2,0) to [out=120,in=240] (2,4);
            	\draw[thick,red] (4,0) to [out=120,in=240] (4,4);
             
                \draw [fill] (0,0) circle [radius=0.1];
                \draw [fill] (0,2) circle [radius=0.1];
                \draw[fill] (0,4) circle [radius=0.1];
                \draw [fill] (2,0) circle [radius=0.1];
                \draw [fill] (2,2) circle [radius=0.1];
                \draw[fill] (2,4) circle [radius=0.1];
                \draw [fill] (4,0) circle [radius=0.1];
                \draw [fill] (4,2) circle [radius=0.1];
                \draw[fill] (4,4) circle [radius=0.1];
                \node[above right] at (0,4) {$x_{32}$};
                \node[above right] at (2,4) {$x_{42}$};
                \node[above right] at (4,4) {$x_{52}$};
                \node[above right] at (0,2) {$x_{24}$};
                \node[above right] at (2,2) {$x_{25}$};
                \node[above right] at (4,2) {$x_{23}$};
                \node[above right] at (0,0) {$x_{34}$};
                \node[above right] at (2,0) {$x_{45}$};
                \node[above right] at (4,0) {$x_{53}$};
                \node[draw=none,fill=none] at(-2,-2){};
                \node[draw=none,fill=none] at(6,6){};
            \end{tikzpicture}
        \end{adjustbox}
    	\caption{}
    	\label{fig:neu:preswitch}
    \end{subfigure}
    \begin{subfigure}[b]{0.45\textwidth}
		\centering
    		\begin{adjustbox}{width=0.95\textwidth}
        		\begin{tikzpicture}[scale=0.7]
                \clip (-2,-2) rectangle (6,6);
                \draw (0,0) -- (4,0);
                \draw (0,2) -- (4,2);
                \draw (0,4) -- (4,4);
                \draw (0,0) to [out=330,in=210] (4,0);
                \draw (0,2) to [out=330,in=210] (4,2);
                \draw (0,4) to [out=330,in=210] (4,4);
        	    \draw[red,thick] (4,0) -- (0,4);
        	    \draw[red,thick] (4,2) -- (2,4);
        	      \draw[red,thick] (0,2) -- (2,0);
        	    \draw[thick,red] (4,0) to [out=165,in=285] (0,4);
                \draw[thick,red] (0,0) to [out=135,in=135,looseness=2.5] (2,4);
                \draw[thick,red] (0,0) to [out=-45,in=-45,looseness=2.5] (4,2);
                \draw[thick,red] (4,4) to [out=135,in=135,looseness=2.5] (0,2);
                \draw[thick,red] (4,4) to [out=-45,in=-45,looseness=2.5] (2,0);
                \draw [fill] (0,0) circle [radius=0.1];
                \draw [fill] (0,2) circle [radius=0.1];
                \draw[fill] (0,4) circle [radius=0.1];
                \draw [fill] (2,0) circle [radius=0.1];
                \draw [fill] (2,2) circle [radius=0.1];
                \draw[fill] (2,4) circle [radius=0.1];
                \draw [fill] (4,0) circle [radius=0.1];
                \draw [fill] (4,2) circle [radius=0.1];
                \draw[fill] (4,4) circle [radius=0.1];
                \node[above right] at (0,4) {$x_{32}$};
                \node[above right] at (2,4) {$x_{42}$};
                \node[above right] at (4,4) {$x_{52}$};
                \node[above right] at (0,2) {$x_{24}$};
                \node[above right] at (2,2) {$x_{25}$};
                \node[above right] at (4,2) {$x_{23}$};
                \node[above right] at (0,0) {$x_{34}$};
                \node[above right] at (2,0) {$x_{45}$};
                \node[above right] at (4,0) {$x_{53}$};
                \node[draw=none,fill=none] at(-2,-2){};
                \node[draw=none,fill=none] at(6,6){};
            \end{tikzpicture}
    	\end{adjustbox}
		\caption{}
		\label{fig:neu:afterswitch}
	\end{subfigure}
    \caption{Replacing the red edges on the left by those on the right, transforms~$\G^{(L)}$ into the Neumaier graph~$\G^{(25)}$.}
    \label{fig:neu:switching}
 \end{figure}
 
 The adjacency matrix of the new strictly Neumaier graph~$\G^{(25)}$ is a symmetric block matrix $M'$ given by
\begin{align*}
    M'= \begin{pNiceMatrix}
        0 & \bm{j} & \bm{j} & \bm{j} & \bm{j} & \bm{0} & \bm{0} & \bm{0} & \bm{0}\\
        \bm{j}^t & B & E_1 & E_2 & E_3 & \overline{E_1} & \overline{E_2} & \overline{E_3} & O\\
         \bm{j}^t & E_1^t & B & I & I & I+A& A & I & B\\
         \bm{j}^t & E_2^t & I & B & I & I & I+A & A & I+A^{t}\\
         \bm{j}^t & E_3^t & I & I & B & A & I & I+A & I+A\\
         \bm{0}^t & \overline{E_1}^t & I+A^t & I & A^t & B & \Block[fill=black!15,rounded-corners]{1-2}{}I & \Block[fill=black!15,rounded-corners]{2-1}{} I & I+A^{t}\\
        \bm{0}^t & \overline{E_2}^t & A^t & I+A^t & I& \Block[fill=black!15,rounded-corners]{2-1}{}I & B & I & I+A\\
        \bm{0}^t &\overline{E_3}^t & I & A^t & I+A^t & \Block[fill=black!15,rounded-corners]{1-2}{}I & I & B & B\\
        \bm{0}^t & O & B & I+A & I+A^t & I+A & I+A^t & B & O
    \end{pNiceMatrix}.
\end{align*}

 \noindent It is obtained from~$M$ by replacing the highlighted blocks with identity matrices. Although only nine edges are switched, nine 5-cliques in~$\G^{(L)}$ are destroyed by the switching, leaving only six in~$\G^{(25)}$. This matrix representation allows for a direct check that~$\G^{(25)}$ is indeed strictly Neumaier.
 
 The new strictly Neumaier graph can be obtained in Sage using graph6 identifier \[\text{`X|\}UJk]J\}I?tYBKNrZ|GqQhebAybQwwh\}KiWWdUHiKJySWvQ\string^?\}'}. \] 
 The automorphism group~$\Aut(\G^{(25)})$ has order eighteen and is structurally isomorphic to~$(C_3\times C_3)\rtimes C_2$. It has five orbits on the vertices:~$\{x_{11}\}$, $L'\setminus\{x_{11}\}$, $(C_1\cup C_2\cup C_3)\setminus L'$, $(C'_1\cup C'_2\cup C'_3)\setminus L'$ and $\{x_{35},x_{43},x_{54}\}$. The group $\Aut(\G^{(25)})$ is a subgroup of~$\Aut(\G^{(L)})$, which has size 72, is structurally isomorphic to $(C_3\times A_4)\rtimes C_2$ and which has only three orbits on the vertices.
 
	\section{A new approach for nonexistence}
	\label{sec:neu:nonexistence}
	In this section we propose a novel approach, which makes uses of integer linear programming (ILP), to show that there can be no strictly Neumaier graph with parameters~$(25,16,9;3,5)$, $(28,18,11;4,7)$, $(33,24,17;6,9)$, $(35,22,12;3,5)$, $(40,30,22;7,10)$ or~$(55,34,18;3,5)$. 	
	\subsection{Integer linear programming and the graph complement}\label{sec:ip}
	The parameter sets that we will discard using the ILP model all have high degree, so it makes sense to instead work with the complement graph. Note that computationally this does not make a difference, as this is a matter of choosing edges versus non-edges. However, to speed up the search, we will make some theoretical observations, which are easier to derive on the complement graphs since they are sparser.
	
	Suppose that~$\Gamma$ is a graph such that~$\overline{\Gamma}$ is strictly Neumaier with parameters~$(v,k,\allowbreak \lambda; e,s)$. Then~$\Gamma$ is $k'$-regular, $\mu'$-co-edge-regular and contains a $e'$-regular coclique~$C$ of size $s$, where $k'= v-k-1$, $\mu'= v-2k+\lambda$ and~$e'=s-e$ respectively. Since the complement~$\overline{\Gamma}$ is strictly Neumaier,~$\Gamma$ cannot be $\lambda'$-edge-regular, with~$\lambda'= v-2-2k+\frac{k(k-\lambda-1)}{v-k-1}$; note that $\Gamma$ cannot be $\ell$-edge-regular for any $\ell\neq\lambda'$.
	
	Using the parameters of this complement graph, we will convert the problem of finding a strictly Neumaier graph to an integer linear program. As we are only interested in the existence of a feasible solution, there is no objective function. The full formulation of the ILP is shown in Figure~\ref{fig:neu:ILP}. We give a more detailed explanation of each of the variables and constraints below. 

 Assume without loss of generality that the vertices are labeled~$1,2,\dots,v$ and that~$[s]:=\{1,2,\dots,s\}$ is the~$e'$-regular coclique. Every unordered pair of vertices~$\{u,w\}$ is associated with a binary variable~$x_{uw}$, which symbolizes whether edge~$\{u,w\}$ is in the graph. In addition, there is a variable~$y_{tuw}$ for every pair~$\{u,w\}$ and vertex~$t$, which equals one whenever the edges~$\{t,u\}$ and~$\{t,w\}$ exist, i.e., if~$t$ is a common neighbor of~$u$ and~$w$. Constraints~\ref{fig:neu:ILP}(\ref{con:neu:y1})-(\ref{con:neu:y3}) ensure that~$y_{tuw}$ is consistent with the values of~$x_{tu}$ and~$x_{tw}$.
 
The constructed graph should be~$k'$-regular, hence the number of edges containing a given vertex must equal~$k'$ (constraint~\ref{fig:neu:ILP}(\ref{con:neu:degree})).
 Constraints~\ref{fig:neu:ILP}(\ref{con:neu:cer1})-(\ref{con:neu:cer2}) force the solution to be co-edge-regular by ensuring that whenever vertices~$u,w$ are not adjacent, their number of common neighbors is exactly~$\mu'$. Note that if~$x_{uw}=1$, the number of common neighbors is bounded between~$0$ and~$v$ by these constraints, so if~$v\sim w$ there is no restriction.
 The regular coclique is induced by constraints~\ref{fig:neu:ILP}(\ref{con:neu:coclique1}) and~\ref{fig:neu:ILP}(\ref{con:neu:coclique2}), which prevent edges between coclique vertices and restrict the number of neighbors of each vertex~$u\in [v]\setminus [s]$ into~$[s]$ respectively. 
 
 As we are considering strictly Neumaier graphs, the output graph should not be~$\lambda'$-edge-regular, so there exists a pair of adjacent vertices~$\{p,q\}$ that do not have exactly~$\lambda'$-common neighbors. Without loss of generality, we may assume that this pair is either~$\{s,s+1\}$ or~$\{s+1,s+2\}$. We therefore solve the ILP four times: for each pair, once with a constraint forcing the pair to have more than~$\lambda'$ common neighbors (constraint \ref{fig:neu:ILP}(\ref{con:neu:er2})), and once such that it has fewer (constraint \ref{fig:neu:ILP}(\ref{con:neu:er3})).

For some parameter sets considered in this article, we will be able to prove that certain edges must exist in the complement of a strictly Neumaier graph. In this case, we can add the set~$S$ of these edges as an input to the ILP and add constraints to ensure that these edges exist in any solution. If we do so, we might need to run the ILP more than four times to ensure that the graph is not~$\lambda'$-edge-regular, because the automorphism group of the subgraph~$\Gamma'= (V,S)$ may induce more than two orbits of potentially adjacent vertex pairs. 
	
	\begin{figure}
		\centering
		\boxed{
        {\footnotesize
			{\setstretch{1.2}
				\begin{tabular}{llll}
					{\tt input:} & \text{parameters } $v$,$k'$,$\lambda'$,$\mu'$,$e'$,$s$ && \\ & pair $\{p,q\}\in \binom{[v]}{2}$ &&\\
                    &set of fixed edges $S\subset \binom{[v]}{2}$&&\\
					{\tt variables:}&$x_{uw}\in \{0,1\}$& $\{u,w\}\in \binom{[v]}{2}$&\\
					&$y_{tuw}\in \{0,1\}$& $\{u,w\}\in \binom{[v]}{2},\ t\in [v]\setminus \{u,w\}$&\\
					{\tt constraints:} & degree:&&\\
					& $\sum_{w\in [v]\setminus u} x_{uw}  = k'$ & $u \in [v]$ &\constraintlabel{con:neu:degree}\\
					&&&\\
					&setting $y$: &&\\
					&$y_{tuw} \ge x_{tu} + x_{tw} - 1$ & $\{u,w\}\in \binom{[v]}{2},\ t\in [v]\setminus \{u,w\}$&\constraintlabel{con:neu:y1}\\
					&$y_{tuw} \le x_{tu}$ & $\{u,w\}\in \binom{[v]}{2},\ t\in [v]\setminus \{u,w\}$&\constraintlabel{con:neu:y2}\\
					&$y_{tuw} \le x_{tw}$ & $\{u,w\}\in \binom{[v]}{2},\ t\in [v]\setminus \{u,w\}$&\constraintlabel{con:neu:y3}\\
					&&&\\
					&co-edge-regularity:&&\\
					&$\sum_{t\in [v]\setminus\{u,w\}} y_{tuw}\ge \mu'(1-x_{uw})$& $\{u,w\}\in \binom{[v]}{2}$&\constraintlabel{con:neu:cer1}\\
					&$\sum_{t\in [v]\setminus\{u,w\}} y_{tuw}\le \mu' + (v-\mu') x_{uw}$& $\{u,w\}\in \binom{[v]}{2}$&\constraintlabel{con:neu:cer2}\\
					&&&\\
					&regular coclique:&&\\
					&$x_{uw} = 0$& $\{u,w\}\in \binom{[s]}{2}$&\constraintlabel{con:neu:coclique1}\\
					&$\sum_{w\in [s]}x_{uw} = e'$& $u\in [v]\setminus [s]$&\constraintlabel{con:neu:coclique2}\\
					&&&\\
					&fixed edges:&&\\
					&$x_{uw} = 1$& $\{u,w\}\in S$&\constraintlabel{con:neu:edges}\\
					&&&\\
					&no edge-regularity:&&\\
					&$x_{pq} = 1$&&\constraintlabel{con:neu:er1}\\
					&$\sum_{t\in [v]\setminus\{p,q\}} y_{tpq} \ge \lambda'+1$\ \ \  \text{or}& &\constraintlabel{con:neu:er2}\\
					&$\sum_{t\in [v]\setminus\{p,q\}} y_{tpq} \le \lambda'-1$&&\constraintlabel{con:neu:er3}\\
		\end{tabular}}}}
		\caption{Integer linear program that determines the existence of the complement of a strictly Neumaier graph with parameters~$(v,k,\lambda;e,s)$, i.e., a $k'$-regular, $\mu'$-co-edge-regular graph with a $e'$-regular coclique of size $s$. The given pair of vertices~$\{p,q\}$ is used to prevent~$\lambda'$-edge-regularity by setting its number of common neighbors strictly smaller or larger than~$\lambda'$.}
		\label{fig:neu:ILP}
	\end{figure}

	\subsection{Nonexistence results of small strictly Neumaier graphs}
	Using the integer program from Section \ref{sec:ip}, we show that there can be no strictly Neumaier graph with parameters $(25,16,9;3,5)$, $(28,18,11;4,7)$, $(33,24,17;\allowbreak 6,9)$, $(35,22,12;3,5)$, $(40,30,22;7,10)$ and $(55,34,18;3,5)$. We run the program using the Gurobi solver~\cite{gurobi} in Python on a single core of an Intel Xeon Platinum 8260 CPU (2.4 GHz) using 16 threads. To speed up the computations, we derive additional theoretical and computer-aided constraints regarding the structure of a potential strictly Neumaier graph for each specific parameter set. 
	
	\begin{case}{1}$(25,16,9;3,5)$.
	\end{case}
	\noindent Suppose that~$\G$ is a strictly Neumaier graph with parameters $(25,16,9;3,5)$. We will use the following computer-aided result on the intersections of its regular cliques.
	
	\begin{lemma}\label{lem:neu:2512_int3}
		Suppose that~$\Gamma$ is a Neumaier graph with parameters~$(25,16,9;3,5)$ and a pair of~$5$-cliques~$S_1$ and~$S_2$ such that~$|S_1\cap S_2| = 3$. Then~$\Gamma$ is strongly regular.
	\end{lemma}
	\begin{proof}
		Let~$u\in S_1\cap S_2$. By assumption,~$\Gamma(u)$ contains two 4-cliques,~$S'_1$ and~$S'_2$, which are 2-regular within~$\Gamma(u)$ and intersect in two vertices. These two vertices both have nine neighbors in $\Gamma(u)$ and hence six neighbors in $\Gamma_2(u)$. It follows that they have at least four common neighbors in~$\Gamma_2(u)$; moreover they have at least five common neighbors in~$\Gamma(u)\cup \{u\}$, namely $\{u\}\cup(S_{1}\triangle S_{2})$. Since they have nine common neighbors in total, they must have exactly five in~$\Gamma(u)\cup \{u\}$ and four in~$\Gamma_2(u)$. This leaves two vertices of~$\Gamma(u)$ which are adjacent to neither, but still have two neighbors in each clique. Therefore, they are both adjacent to all of~$S'_1\triangle S'_2$. Figure~\ref{fig:neu:ng2516case3_nbhood} shows the resulting partial neighborhood of~$u$. We will adapt the vertex labeling in the figure for the remainder of the proof, hence we have~$S'_{1}=\{x_{5},x_{6},x_{10},x_{11}\}$ and~$S'_{2}=\{x_{6},x_{7},x_{11},x_{12}\}$.
		
		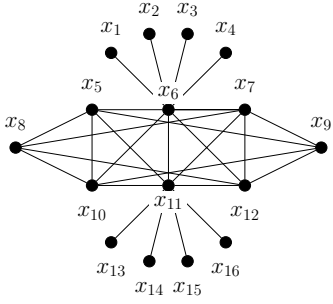
\begin{figure}[!htb]\centering
			\begin{adjustbox}{width=0.3\textwidth}
				\begin{tikzpicture}[nodes={draw,circle,fill,inner sep=3pt}]
					\node[label=above:{\Large$x_1$}] (x1) at (-1.5,1.5){};
					\node[label=above:{\Large$x_2$}] (x2) at (-0.5,2){};
					\node[label=above:{\Large$x_3$}] (x3) at (0.5,2){};
					\node[label=above:{\Large$x_4$}] (x4) at (1.5,1.5){};
					\node[label=above:{\Large$x_5$}] (x5) at (-2,0){};
					\node[label={[fill=white,rectangle]above:{\Large$x_6$}}] (x6) at (0,0){};
					\node[label=above:{\Large$x_7$}] (x7) at (2,0){};
					\node[label=above:{\Large$x_8$}] (x8) at (-4,-1){};
					\node[label=above:{\Large$x_9$}] (x9) at (4,-1){};
					\node[label=below:{\Large$x_{10}$}] (x10) at (-2,-2){};
					\node[label={[fill=white,rectangle]below:{\Large$x_{11}$}}] (x11) at (0,-2){};
					\node[label=below:{\Large$x_{12}$}] (x12) at (2,-2){};
					\node[label=below:{\Large$x_{13}$}] (x13) at (-1.5,-3.5){};
					\node[label=below:{\Large$x_{14}$}] (x14) at (-0.5,-4){};
					\node[label=below:{\Large$x_{15}$}] (x15) at (0.5,-4){};
					\node[label=below:{\Large$x_{16}$}] (x16) at (1.5,-3.5){};
					
					\begin{pgfonlayer}{bg}    
						\draw (x5) -- (x10) -- (x6) -- (x11) -- (x5) -- (x6) -- (x7) -- (x11) -- (x12) -- (x6) -- (x7) -- (x12);
						\draw (x10) -- (x11);
						\draw (x5) -- (x8) -- (x10) -- (x9) -- (x5);
						\draw (x7) -- (x8) -- (x12) -- (x9) -- (x7);
						\draw (x1) -- (x6) -- (x2);
						\draw (x3) -- (x6) -- (x4);
						\draw (x13) -- (x11) -- (x14);
						\draw (x15) -- (x11) -- (x16);
					\end{pgfonlayer}
				\end{tikzpicture}
			\end{adjustbox}
			\caption{Required edges in the neighborhood of vertex $u$.}
			\label{fig:neu:ng2516case3_nbhood}
		\end{figure}
		
		Vertices~$x_7$ and~$x_{12}$ both must be adjacent to two distinct vertices in~$\{x_1,\dots,x_4\}$ and to two distinct vertices in~$\{x_{13},\dots,x_{16}\}$, else~$S'_{2}$ would not be 2-regular or they would have at least five common neighbors with~$x_{6}$ or $x_{11}$. Without loss of generality, we may assume~$x_7\sim x_3,x_4,x_{13},x_{14}$ and~$x_{12}\sim x_1,x_2,x_{15},x_{16}$. The same holds for~$x_5$ and~$x_{10}$. This leads to the four non-isomorphic graphs shown in Figure~\ref{fig:neu:ng2516case3_nbhoods}.
		
		\begin{figure}[htb!]\centering
			\begin{subfigure}[b]{0.22\textwidth}
				\centering
				\begin{adjustbox}{width=0.95\textwidth}
					\begin{tikzpicture}[nodes={draw,circle,fill,inner sep=3pt}]
						\node (x1) at (-1.5,1.5){};
						\node (x2) at (-0.5,2){};
						\node (x3) at (0.5,2){};
						\node (x4) at (1.5,1.5){};
						\node (x5) at (-2,0){};
						\node (x6) at (0,0){};
						\node (x7) at (2,0){};
						\node (x8) at (-4,-1){};
						\node (x9) at (4,-1){};
						\node (x10) at (-2,-2){};
						\node (x11) at (0,-2){};
						\node (x12) at (2,-2){};
						\node (x13) at (-1.5,-3.5){};
						\node (x14) at (-0.5,-4){};
						\node (x15) at (0.5,-4){};
						\node (x16) at (1.5,-3.5){};
						
						\begin{pgfonlayer}{bg}    
							\draw (x5) -- (x10) -- (x6) -- (x11) -- (x5) -- (x6) -- (x7) -- (x11) -- (x12) -- (x6) -- (x7) -- (x12);
							\draw (x10) -- (x11);
							\draw (x5) -- (x8) -- (x10) -- (x9) -- (x5);
							\draw (x7) -- (x8) -- (x12) -- (x9) -- (x7);
							\draw (x1) -- (x6) -- (x2);
							\draw (x3) -- (x6) -- (x4);
							\draw (x13) -- (x11) -- (x14);
							\draw (x15) -- (x11) -- (x16);
							
							\draw (x4) -- (x7) -- (x3);
							\draw (x7) -- (x14);
							\draw (x7) to [out=-125,in=35, looseness=1] (x13);
							\draw (x12) to [out=125,in=-35, looseness=1] (x1);
							\draw (x12) -- (x2);
							\draw (x15) -- (x12) -- (x16);
							
							\draw (x1) -- (x5) -- (x2);
							\draw (x13) -- (x5) -- (x14);
							\draw (x3) -- (x10) to [out=55,in=-145, looseness=1] (x4);
							\draw (x15) -- (x10) -- (x16);
						\end{pgfonlayer}
					\end{tikzpicture}
				\end{adjustbox}
				\caption{}
				\label{fig:neu:ng2516nbhoods_case3_sub1}
			\end{subfigure}
			\begin{subfigure}[b]{0.22\textwidth}
				\centering
				\begin{adjustbox}{width=0.95\textwidth}
					\begin{tikzpicture}[nodes={draw,circle,fill,inner sep=3pt}]
						\node (x1) at (-1.5,1.5){};
						\node (x2) at (-0.5,2){};
						\node (x3) at (0.5,2){};
						\node (x4) at (1.5,1.5){};
						\node (x5) at (-2,0){};
						\node (x6) at (0,0){};
						\node (x7) at (2,0){};
						\node (x8) at (-4,-1){};
						\node (x9) at (4,-1){};
						\node (x10) at (-2,-2){};
						\node (x11) at (0,-2){};
						\node (x12) at (2,-2){};
						\node (x13) at (-1.5,-3.5){};
						\node (x14) at (-0.5,-4){};
						\node (x15) at (0.5,-4){};
						\node (x16) at (1.5,-3.5){};
						
						\begin{pgfonlayer}{bg}    
							\draw (x5) -- (x10) -- (x6) -- (x11) -- (x5) -- (x6) -- (x7) -- (x11) -- (x12) -- (x6) -- (x7) -- (x12);
							\draw (x10) -- (x11);
							\draw (x5) -- (x8) -- (x10) -- (x9) -- (x5);
							\draw (x7) -- (x8) -- (x12) -- (x9) -- (x7);
							\draw (x1) -- (x6) -- (x2);
							\draw (x3) -- (x6) -- (x4);
							\draw (x13) -- (x11) -- (x14);
							\draw (x15) -- (x11) -- (x16);
							
							\draw (x4) -- (x7) -- (x3);
							\draw (x7) -- (x14);
							\draw (x7) to [out=-125,in=35, looseness=1] (x13);
							\draw (x12) to [out=125,in=-35, looseness=1] (x1);
							\draw (x12) -- (x2);
							\draw (x15) -- (x12) -- (x16);
							
							\draw (x1) -- (x5) -- (x2);
							\draw (x13) -- (x5) -- (x15);
							\draw (x3) -- (x10) to [out=55,in=-145, looseness=1] (x4);
							\draw (x14) -- (x10) -- (x16);
						\end{pgfonlayer}
					\end{tikzpicture}
				\end{adjustbox}
				\caption{}
				\label{fig:neu:ng2516nbhoods_case3_sub2}
			\end{subfigure}
			\begin{subfigure}[b]{0.22\textwidth}
				\centering
				\begin{adjustbox}{width=0.95\textwidth}
					\begin{tikzpicture}[nodes={draw,circle,fill,inner sep=3pt}]
						\node (x1) at (-1.5,1.5){};
						\node (x2) at (-0.5,2){};
						\node (x3) at (0.5,2){};
						\node (x4) at (1.5,1.5){};
						\node (x5) at (-2,0){};
						\node (x6) at (0,0){};
						\node (x7) at (2,0){};
						\node (x8) at (-4,-1){};
						\node (x9) at (4,-1){};
						\node (x10) at (-2,-2){};
						\node (x11) at (0,-2){};
						\node (x12) at (2,-2){};
						\node (x13) at (-1.5,-3.5){};
						\node (x14) at (-0.5,-4){};
						\node (x15) at (0.5,-4){};
						\node (x16) at (1.5,-3.5){};
						
						\begin{pgfonlayer}{bg}    
							\draw (x5) -- (x10) -- (x6) -- (x11) -- (x5) -- (x6) -- (x7) -- (x11) -- (x12) -- (x6) -- (x7) -- (x12);
							\draw (x10) -- (x11);
							\draw (x5) -- (x8) -- (x10) -- (x9) -- (x5);
							\draw (x7) -- (x8) -- (x12) -- (x9) -- (x7);
							\draw (x1) -- (x6) -- (x2);
							\draw (x3) -- (x6) -- (x4);
							\draw (x13) -- (x11) -- (x14);
							\draw (x15) -- (x11) -- (x16);
							
							\draw (x4) -- (x7) -- (x3);
							\draw (x7) -- (x14);
							\draw (x7) to [out=-125,in=35, looseness=1] (x13);
							\draw (x12) to [out=125,in=-35, looseness=1] (x1);
							\draw (x12) -- (x2);
							\draw (x15) -- (x12) -- (x16);
							
							\draw (x1) -- (x5) -- (x3);
							\draw (x13) -- (x5) -- (x15);
							\draw (x2) -- (x10) to [out=55,in=-145, looseness=1] (x4);
							\draw (x14) -- (x10) -- (x16);
						\end{pgfonlayer}
					\end{tikzpicture}
				\end{adjustbox}
				\caption{}
				\label{fig:neu:ng2516nbhoods_case3_sub3}
			\end{subfigure}
			\begin{subfigure}[b]{0.22\textwidth}
				\centering
				\begin{adjustbox}{width=0.95\textwidth}
					\begin{tikzpicture}[nodes={draw,circle,fill,inner sep=3pt}]
						\node (x1) at (-1.5,1.5){};
						\node (x2) at (-0.5,2){};
						\node (x3) at (0.5,2){};
						\node (x4) at (1.5,1.5){};
						\node (x5) at (-2,0){};
						\node (x6) at (0,0){};
						\node (x7) at (2,0){};
						\node (x8) at (-4,-1){};
						\node (x9) at (4,-1){};
						\node (x10) at (-2,-2){};
						\node (x11) at (0,-2){};
						\node (x12) at (2,-2){};
						\node (x13) at (-1.5,-3.5){};
						\node (x14) at (-0.5,-4){};
						\node (x15) at (0.5,-4){};
						\node (x16) at (1.5,-3.5){};
						
						\begin{pgfonlayer}{bg}    
							\draw (x5) -- (x10) -- (x6) -- (x11) -- (x5) -- (x6) -- (x7) -- (x11) -- (x12) -- (x6) -- (x7) -- (x12);
							\draw (x10) -- (x11);
							\draw (x5) -- (x8) -- (x10) -- (x9) -- (x5);
							\draw (x7) -- (x8) -- (x12) -- (x9) -- (x7);
							\draw (x1) -- (x6) -- (x2);
							\draw (x3) -- (x6) -- (x4);
							\draw (x13) -- (x11) -- (x14);
							\draw (x15) -- (x11) -- (x16);
							
							\draw (x4) -- (x7) -- (x3);
							\draw (x7) -- (x14);
							\draw (x7) to [out=-125,in=35, looseness=1] (x13);
							\draw (x12) to [out=125,in=-35, looseness=1] (x1);
							\draw (x12) -- (x2);
							\draw (x15) -- (x12) -- (x16);
							
							\draw (x1) -- (x5) -- (x3);
							\draw (x13) -- (x5) -- (x14);
							\draw (x2) -- (x10) to [out=55,in=-145, looseness=1] (x4);
							\draw (x15) -- (x10) -- (x16);
						\end{pgfonlayer}
					\end{tikzpicture}
				\end{adjustbox}
				\caption{}
				\label{fig:neu:ng2516nbhoods_case3_sub4}
			\end{subfigure}
			
			\caption{The neighborhood of $u$ must have one of these four non-isomorphic graphs as a subgraph.} 
			\label{fig:neu:ng2516case3_nbhoods}
		\end{figure}
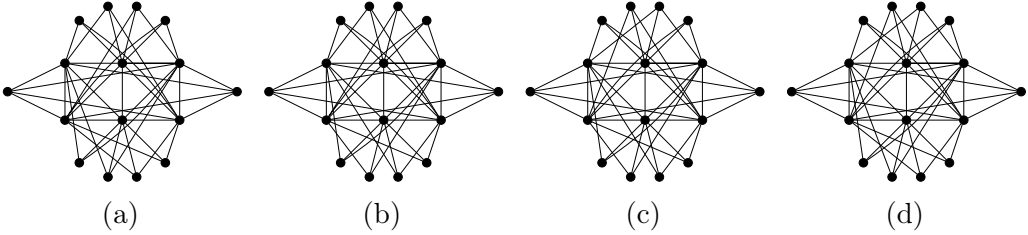
		
		To complete each partial neighborhood, we apply a recursive brute-force procedure in Sage. For a given vertex~$x_i$ with degree less than nine, we compute the set~$S$ of potential neighbors. Vertices in~$S$ must have degree at most eight and have at most four common neighbors with~$x_i$ so far. If~$9-d_{x_i} > |S|$, no valid set of neighbors can be found for~$x_i$ and the algorithm returns to the previous iteration. Otherwise, it computes the orbits of all~$(9-d_{x_i})$-subsets of~$S$ under the automorphism group of the partial neighborhood. We choose one representative of each orbit, add the corresponding edges to a copy of the graph and recurse on each of these new partial neighborhoods.  
		
		After completing the potential neighborhoods of~$u$, we apply the same idea to generate all possible edges between~$\Gamma(u)$ and~$\Gamma_2(u)$. For each vertex~$w\in \Gamma(u)$, we generate all possible sets of neighbors of~$w$ in~$\Gamma_2(u)$ and recurse on one representative of each orbit under the automorphism group of the partial graph we have constructed so far, pruning whenever edge-regularity is violated. We then apply a similar procedure to complete~$\Gamma_2(w)$. The results are summarized in Table~\ref{tab:neu:ng2516case3}, which shows how many options remain after applying each algorithm to the four graphs in Figure~\ref{fig:neu:ng2516case3_nbhoods}. We see that only one Neumaier graph is found, which is isomorphic to the graph generated by the unique pair (up to isotopy) of mutually orthogonal~$5\times 5$ Latin squares. This graph is known to be strongly regular.
	\end{proof}
	
	\begin{table}[!htb]\centering
		\begin{tabular}{l|c|c|c}
			Graph & Options~$\Gamma(u)$ & 
			Options~$E(\Gamma(u),\Gamma_2(u))$ & Options~$\Gamma_2(u)$ \\\hline
			\ref{fig:neu:ng2516nbhoods_case3_sub1}&38938&0&-\\
			\ref{fig:neu:ng2516nbhoods_case3_sub2}&34813&0&-\\
			\ref{fig:neu:ng2516nbhoods_case3_sub3}&241467&108&0\\
			\ref{fig:neu:ng2516nbhoods_case3_sub4}&239500&3&1
		\end{tabular}
		\caption{The number of options left after applying each algorithm to the graphs in Figure \ref{fig:neu:ng2516case3_nbhoods}.
    }
		\label{tab:neu:ng2516case3}
	\end{table}

	The complement of~$\G$ has parameters $(v,k',\mu';e',s) = (25,8,2;2,5)$ and is not 3-edge-regular. Let~$C$ be the~$e'$-regular coclique in this graph. Each of the ten pairs of vertices in~$C$ has two common neighbors in~$V\setminus C$. As~$|V\setminus C| = 20$ and every~$w\in V\setminus C$ has exactly two neighbors in~$C$, this means that every pair of coclique vertices has two distinct common neighbors, and these sets of two vertices partition $V\setminus C$. This fixes the edges into~$C$ up to symmetry. Moreover, two common neighbors corresponding to the same pair must be adjacent. If not, we could replace the pair with these neighbors and find a coclique intersecting~$C$ in three vertices. Then~$\overline{\Gamma}$ has two intersecting maximum cliques, which contradicts Lemma~\ref{lem:neu:2512_int3}. Combined, these observations fix fifty of the hundred edges in~$\overline{\Gamma}$. This creates four orbits of (potentially) adjacent pairs of vertices. We run the ILP twice for each pair, once with constraint  \ref{fig:neu:ILP}(\ref{con:neu:er2}) and once with constraint \ref{fig:neu:ILP}(\ref{con:neu:er3}). In all cases, the ILP does not have a solution, hence there is no strictly Neumaier graph with parameters~$(25,16,9;3,5)$.
	
	\begin{case}{2}$(28,18,11;4,7)$.
	\end{case}
    \noindent The complement of a strictly Neumaier graph with these parameters would have~$k'= 9$,~$\mu'= 3$ and~$e'= 3$. The 21 vertices outside the regular coclique~$C$ induce a 2-(7,3,3) design (possibly with repeated blocks) on the vertices in~$C$, where every block corresponds to the set of neighbors of a vertex~$v\in V\setminus C$. There exist ten 2-(7,3,3) designs (with repeated blocks allowed) \cite[Table II.1.20]{colbourn2007crc}. This gives us ten sets of fixed edges as input for the ILP. Moreover, we may discard constraints \ref{fig:neu:ILP}(\ref{con:neu:er1})-(\ref{con:neu:er3}), since there exist no strongly regular graphs on these parameters.
    In all cases, the ILP has no solution. 

    \begin{case}{3}$(33,24,17;6,9)$.
	\end{case}
    \noindent The complement of a strictly Neumaier graph with these parameters has~$k'= 8$,~$\mu'= 2$ and~$e'= 3$. Once again, there exist no strongly regular graphs on the given parameters, so we discard constraints \ref{fig:neu:ILP}(\ref{con:neu:er1})-(\ref{con:neu:er3}) in the ILP. In the complement graph, the 24 vertices outside the regular coclique~$C$ induce a 2-(9,3,2) design (possibly with repeated blocks) on the vertices in~$C$, where every block corresponds to the set of neighbors of a vertex~$v\in V\setminus C$. There exist 36 such designs up to isomorphism~\cite[Table II.1.23]{colbourn2007crc}, and each of these gives rise to a set of fixed edges as input for the ILP. In all 36 cases, the ILP does not have a solution.
		
	\begin{case}{4}$(35,22,12;3,5)$ and~$(55,34,18;3,5)$.
	\end{case}
 \noindent Suppose there exists a strictly Neumaier graph with parameters~$(35,22,12;3,5)$. Then the complement graph is 12-regular, 3-co-edge-regular and has a 2-regular coclique of size 5. Since there exists no 12-regular strongly regular graphs on 35 vertices, we do not need to take into account the edge-regularity constraints in the ILP. Moreover, the 10 pairs of nonadjacent vertices in the coclique each have three common neighbors. Since the coclique is 2-regular, these ten sets of common neighbors partition the vertices outside the coclique. This means that, up to symmetry, there is only one possible configuration of edges between the coclique and the rest of the complement graph. There exist no strongly regular graphs on the given parameters, so we run the ILP without constraints \ref{fig:neu:ILP}(\ref{con:neu:er1})-(\ref{con:neu:er3}). It has no solutions, hence there exist no strictly Neumaier graph with parameters~$(35,22,12;3,5)$.
 
 The same approach can be used to eliminate $(55,34,18;3,5)$, where the complement would be a 24-regular, 5-co-edge-regular graph. Again, there are ten pairs of coclique vertices and the sets of their common neighbors partition the vertices outside the coclique. This fixes the edges between the coclique and the rest of the complement graph up to isomorphism.
 
 Both parameter sets fit the pattern~$(10\mu+5,6\mu+4,3\mu+3;3,5)$ with~$\mu=3$ and~$\mu=5$ respectively. Moreover, the ILP solver terminates very quickly in both cases, despite the large number of vertices. This suggests that there may be an underlying theoretical reason why these graphs cannot exist for odd values of~$\mu$.

  \begin{case}{5}$(40,30,22;7,10)$.
	\end{case}
    \noindent The complement of a strictly Neumaier graph with these parameters has~$k'= 9$,~$\mu'= 2$ and~$e'= 3$. There again exist no strongly regular graphs on the given parameters. In the complement graph, the vertices outside the regular coclique~$C$ induce a 2-(10,3,2) design (possibly with repeated blocks) on the vertices in~$C$, where every block corresponds to the set of neighbors of a vertex~$v\in V\setminus C$. There exist 960 such designs up to isomorphism~\cite{colbourn1983complete,ganter1978complete}. Each of these gives rise to a set of fixed edges as input for the ILP, none of which produce a Neumaier graph. On some inputs, the ILP terminates within several minutes, however, there are also some extreme cases which require several days. It therefore seems that we are reaching the boundary of what is possible with this method. \newline

The results from this section are summarized in the following theorem.

\begin{theorem}
    There exist no strictly Neumaier graphs with parameters~$(25,16,9;\allowbreak 3,5)$, $(28,18,11;4,7)$, $(33,24,17;6,9)$, $(35,22,12;3,5)$, $(40,30,22;7,10)$ and~$(55,34, \allowbreak 18;\allowbreak 3,5)$. 
    \label{th:neu:nonexistence_overview}
\end{theorem}

	\subsection*{Acknowledgements}
	The authors thank Jack Koolen for inspiring discussions on Section \ref{sec:neu:nonexistence}. The authors would like to thank the referee for the many detailed comments.
    \par Aida Abiad is partially supported by the Dutch Research Council through the grant VI.Vidi.213.085 and by the Research Foundation Flanders through the grant 1285921N. Maarten De Boeck and Sjanne Zeijlemaker were partially supported by the Croatian Science Foundation under the project 5713.

	\subsection*{Statements and Declarations}
	The authors have no relevant financial or non-financial interests to disclose.
 
	\bibliographystyle{abbrv}
	\bibliography{bibliography_rev}{}

\end{document}